\documentclass[12pt]{article}
\usepackage{amsthm}
\usepackage{amsmath}
\usepackage{amssymb}
\usepackage{latexsym}
\usepackage{graphicx}
\usepackage{bm}
\usepackage{longtable}
\usepackage{lscape}
\usepackage{multirow}
\usepackage{empheq}

\usepackage{fancybox,color}
\definecolor{lred}{rgb}{1,0.8,0.8}
\definecolor{lblue}{rgb}{0.8,0.8,1}
\definecolor{dred}{rgb}{0.6,0,0}
\definecolor{dblue}{rgb}{0,0,0.5}
\definecolor{violet}{rgb}{0.5804,0.0000,0.8275}
\definecolor{purple}{rgb}{0.2400,0.5700,0.2500}
\newcommand{\COMM}[2]{{
\ifthenelse{\equal{#1}{AT}}{\color{red}}{
\ifthenelse{\equal{#1}{NI}}{\color{blue}}{
\ifthenelse{\equal{#1}{KT}}{\color{green}}}}
[#1: #2]
}}

\newtheorem{lemma}{Lemma}[section]

\newtheorem{THEO}{Theorem}[section]
\newtheorem{EXAMP}[THEO]{Example}
\newcommand{\examp}{\begin{EXAMP} \rm}
\newcommand{\eexamp}{\end{EXAMP}}
\newtheorem{ALGo}[THEO]{Algorithm}
\newcommand{\algo}{\begin{ALGo} \rm}
\newcommand{\ealgo}{\end{ALGo}}

\def\Inprod#1#2{#1 \bullet #2}

\def\0{\bm{0}}
\def\1{\bm{1}}
\def\2{\bm{2}}
\def\3{\bm{3}}
\def\4{\bm{4}}
\def\5{\bm{5}}
\def\6{\bm{6}}
\def\7{\bm{7}}
\def\8{\bm{8}}
\def\9{\bm{9}}

\def\d{\bm{d}}
\def\e{\bm{e}}

\def\r{\bm{r}}

\def\v{\bm{v}}

\def\x{\bm{x}}

\def\A{\bm{A}}
\def\B{\bm{B}}
\def\C{\bm{C}}

\def\F{\bm{F}}
\def\G{\bm{G}}
\def\H{\bm{H}}
\def\I{\bm{I}}

\def\O{\bm{O}}

\def\Q{\bm{Q}}
\def\R{\bm{R}}

\def\X{\bm{X}}
\def\Y{\bm{Y}}
\def\Z{\bm{Z}}

\def\AC{\mathcal{A}}
\def\BC{\mathcal{B}}
\def\CC{\mathcal{C}}

\def\EC{\mathcal{E}}
\def\FC{\mathcal{F}}

\def\XC{\mathcal{X}}

\def\Real{\mathbb{R}}
\def\coneK{\mathbb{K}}

\def\spaceV{\mathbb{V}}
\def\SymMat{\mathbb{S}}

\def\SymN{\mathbb{N}}
\def\Integer{\mathbb{Z}}

\def\balpha{\bm{\alpha}}
\def\bgamma{\bm{\gamma}}

\def\optr3{\mbox{$\mbox{OPT}_{{\mbox{\scriptsize R}}_3}$}}
\def\dnnplus1{\mbox{$\SymMat^{1+n}_+\cap\SymN^{1+n}$}}

\def\Ibinary{\mbox{$I_{\mbox{\scriptsize bin}}$}}
\def\Ibox{\mbox{$I_{\mbox{\scriptsize box}}$}}

\def\balpha{\mbox{\boldmath $\alpha$}}
\def\bbeta{\mbox{\boldmath $\beta$}}
\def\salpha{\mbox{\scriptsize $\balpha$}}
\def\sbeta{\mbox{\scriptsize $\bbeta$}}
\def\bgamma{\mbox{\boldmath $\gamma$}}

\def\balpha{\bm{\alpha}}
\def\bbeta{\bm{\beta}}
\def\bgamma{\bm{\gamma}}

\def\sAC{\mbox{\scriptsize $\AC$}}

\def\sFC{\mbox{\scriptsize $\FC$}}

\def\Ibinary{\mbox{$I_{\mbox{\scriptsize bin}}$}}
\def\Ibox{\mbox{$I_{\mbox{\scriptsize box}}$}}

\def\Inprod#1#2{\left\langle#1, \, #2\right\rangle}
\def\inprod#1#2{\langle#1, \, #2\rangle}

\usepackage{algorithm}
\usepackage{algorithmic}

\def\Ibinary{I_{\mathrm{bin}}}
\def\Ibox{I_{\mathrm{box}}}

\def\deg{\mathrm{deg}}
\def\supp{\mathrm{supp}}

\def\BBC{H}
\def\MomentBBC{M}

\def\Lip{L_f}
\newtheorem{example}[THEO]{Example}
\usepackage{tikz}
\usetikzlibrary{positioning}
\usepackage{lscape}

\def\matBP{BBCPOP}

\begin{document}

\begin{center}

\begin{Large}

\matBP: 
A Sparse Doubly Nonnegative Relaxation \\ of Polynomial Optimization Problems \\ with Binary, Box and Complementarity Constraints

\end{Large}

%
%
%
\end{center}

%


\setcounter{page}{1}

\noindent

\begin{center}
\vspace{-0.2cm}
$\mbox{N. Ito}^{\star}$, $\mbox{S. Kim}^{\dagger}$, $\mbox{M. Kojima}^{\ddagger}$, $\mbox{A. Takeda}^{\mathsection}$, and
$\mbox{K.-C. Toh}^{\mathparagraph}$ 
\vspace{0.1cm}

March, 2018
\end{center}

\noindent
{\bf Abstract. }
The software package \matBP\ is a MATLAB implementation of a hierarchy of sparse doubly nonnegative (DNN) relaxations 
of a class of polynomial optimization (minimization) problems (POPs) with binary, box and complementarity 
(BBC) constraints.  Given a POP in the class and a relaxation order, 
\matBP\ constructs a simple conic optimization problem (COP), 
which serves as a DNN relaxation of the POP, and then solves the COP by 
applying the bisection and projection (BP) method. 
The COP is expressed  with a linear objective function 
and constraints described as a single hyperplane and two cones, which are
the Cartesian product of positive semidefinite cones  
and a polyhedral cone induced from the 
BBC constraints. \matBP\ aims to compute a tight lower bound  for
the optimal value of a large-scale POP in the class 
that is beyond the comfort zone of existing software packages. 
The robustness, reliability and efficiency of \matBP\ are demonstrated in comparison to the state-of-the-art 
software SDP package SDPNAL+ on randomly generated sparse POPs of degree 2 and 3 with up to a few thousands variables,  
and ones of degree 4, 5, 6. and 8 with up to a few hundred variables. 
Comparison with other BBC POPs that arise from combinatorial optimization problems such
as quadratic assignment problems are also reported.
The software package {\bf \matBP} is 
available at https://sites.google.com/site/bbcpop1/.
\vspace{0.1cm}
  

\vspace{0.2cm}


\vspace{0.1cm}
 
 \begin{small}
\noindent
\parbox[t]{0.5cm}{$\star$}
\parbox[t]{14.9cm}{Department of Mathematical Informatics,
        			The University of Tokyo, Tokyo 113-8656, Japan. 
        			This work was supported by Grant-in-Aid for JSPS Research Fellowship JP17J07365.
			({\tt naoki\_ito{@}mist.i.u-tokyo.ac.jp}).
}

\medskip
 
\noindent
\parbox[t]{0.5cm}{$\dagger$}
\parbox[t]{14.9cm}{Department of Mathematics, Ewha W. University,
52 Ewhayeodae-gil, Sudaemoon-gu, Seoul 03760 Korea. 
The research was supported
by  
NRF 2017-R1A2B2005119.
({\tt skim@ewha.ac.kr}).
}

\medskip

\noindent
\parbox[t]{0.5cm}{$\ddagger$}
\parbox[t]{14.9cm}{
Department of Industrial and Systems Engineering,
Chuo University, Tokyo 112-8551 Japan.
This research was supported by Grant-in-Aid for Scientific Research (A) 26242027.
\\                
({\tt kojimamasakzu@mac.com}).
}

\medskip

\noindent
\parbox[t]{0.5cm}{$\mathsection$}
\parbox[t]{14.9cm}{
 Department of Mathematical Analysis and Statistical Inference, 
The Institute of Statistical Mathematics, 
10-3 Midori-cho, Tachikawa, Tokyo 190-8562, Japan
           The work of this author was supported by Grant-in-Aid for Scientific Research (C), 15K00031.
    ({\tt atakeda{@}ism.ac.jp}).
}

\medskip

\noindent
\parbox[t]{0.5cm}{$\mathparagraph$}
\parbox[t]{14.9cm}{Department of Mathematics and Institute of Operations Research and Analytics, 
	National University of
         Singapore, 10 Lower Kent Ridge Road, Singapore 119076. 
         Research supported in part by the Ministry of Education, Singapore, Academic Research Fund under Grant R-146-000-256-114. 
 ({\tt mattohkc@nus.edu.sg}).
}

\end{small}

\newpage

\section{Introduction}

We introduce a Matlab package \matBP \ for computing a tight lower bound of the
optimal value of large-scale sparse
polynomial optimization problems (POPs) 
with  binary, box and complementarity (BBC) constraints. 
Let $f_0$ be a real valued polynomial function defined on the $n$-dimensional Euclidean space $\Real^n$, 
$\Ibox$ and $\Ibinary$
a partition of $N\equiv\{1,2,\ldots,n\}$, {\it i.e.}, $\Ibox\cup\Ibinary = N$ and 
$\Ibox\cap\Ibinary=\emptyset$, and $\CC$ a family of subsets of $N$. 
\matBP \ finds a lower bound for the optimal value $\zeta^*$ of POPs described as
\begin{eqnarray}
\zeta^* = \min_{\x} \left\{ f_0(\x) \Big|
\begin{array}{l}
x_i\in[0,1] ~ (i\in\Ibox) \ \mbox{(box constraint)}, \\
x_j\in\{0,1\} ~ (j\in\Ibinary) \ \mbox{(binary constraint)}, \\
\prod_{j\in C} x_j = 0 ~ (C \in \CC) \ \mbox{(complementarity constraint)} 
\end{array}
 \right\}. \label{POP0} 
\end{eqnarray}
The above BBC constrained
POP \eqref{POP0} has been widely studied as they have many applications in combinatorial optimization, 
signal processing \cite{GERSHMAN10,LUO10}, transportation
engineering \cite{AAA2016}, and optimal power flow \cite{GHADDAR14,MOLZAHN14}. 

\matBP \ provides a MATLAB implementation to automatically generate a hierarchy of sparse doubly 
nonnegative (DNN) relaxations of 
POP \eqref{POP0} 
together with 
the BP (bisection and projection) method 
as a solver for the resulting DNN relaxations problems.
This software
is developed to find  approximate optimal values of larger-scale POPs 
which are beyond the range that can be comfortably handled by
existing software packages.
More precisely, an approximate optimal value provided by \matBP \ for a POP is a valid lower bound 
for the actual optimal value of the POP
that is generally NP-hard to compute. 
The hierarchy of sparse DNN relaxations implemented can be regarded as 
a variant of the hierarchies of 
sparse SDP relaxations considered in \cite{WAKI2008}.
The BP method was first introduced by Kim, Kojima and Toh in \cite{KIM2013} 
for the dense doubly nonnegative (DNN) relaxation of a class of QOPs such as binary quadratic problems,
maximum clique problems, quadratic multi-knapsack problems, and quadratic assignment problems,
and improved later in \cite{ARIMA2017}. 
In their subsequent work \cite{KIM2016}, the BP method was generalized to handle
the hierarchy of sparse DNN relaxations of a class of binary and box 
constrained POPs. Some numerical results on large-scale binary and box constrained POPs 
and the comparison of the BP method to SparsePOP \cite{WAKI2008} combined with SDPNAL+ \cite{YST2015}
were 
reported in \cite{KIM2016}. In this paper, 
we present an extended version of the BP method for a hierarchy of sparse DNN relaxations of 
a class of 
BBC constrained POPs.

Existing
software packages available for solving general POPs include GloptiPoly \cite{GLOPTIPOLY2003},
 SOSTOOLS \cite{SOSTOOLS} and SparsePOP \cite{WAKI2008}.
Notable numerical methods for POPs that have not been announced as software packages
include (i) the application of the package SDPNAL \cite{SDPNAL} for solving SDP relaxations of POPs
in \cite{NIE2012}; (ii) the application of the solver SDPT3 to solve the
bounded degree sums of squares (BSOS) SDP relaxations in \cite{WEISSER17} and the sparse BSOS relaxations in \cite{WEISSER17}.
GloptiPoly \cite{GLOPTIPOLY2003}, which is designed for the hierarchy of the dense
semidefinite (SDP) relaxations by Lasserre in \cite{LASSERRE2001}, and SOSTOOLS \cite{SOSTOOLS}, which implements
the SDP relaxation by Parrilo \cite{PARRILO2003},
 can handle small-scale dense POPs with at most 20-40 variables. 
By exploiting the structured  sparsity in POPs, 
the hierarchy of 
sparse SDP relaxations was proposed in \cite{WAKI2006} and implemented
as SparsePOP  \cite{WAKI2008}.   
It was shown that SparsePOP could solve medium-scale  general POPs 
of degree up to~4, and unconstrained POPs with 5000 variables \cite{SparsePOP_UG} in less than a minute 
if the sparsity in POPs can be characterized as a banded sparsity pattern 
such as in the
minimization of the chained wood and chained singular functions. More recently, 
BSOS \cite{WEISSER17} and its sparse version of BSOS \cite{WEISSER17} based on a bounded
sums of squares of polynomial have been introduced, and it was 
shown that sparse BSOS \cite{WEISSER17} could solve
POPs with up to 1000 variables for the same 
examples.
We should note that 
solving large-scale unconstrained minimization of such  functions is much easier than solving constrained POPs.
In \cite{NIE2012}, it was demonstrated that  SDPNAL \cite{SDPNAL} can solve
POPs 
of degree up to 6 and 100 variables could be solved by assuming block diagonal structure sparsity. 

Despite  efforts to solve large-scale POPs by proposing new theoretical frameworks and various 
numerical techniques,
large-scale  POPs still remain 
very challenging
to solve. This difficulty arises from solving
large-scale SDP relaxations by SDP solvers, for instance, SeDuMi \cite{STURM99}, SDPA \cite{SDPAv7}, 
SDPT3 \cite{TOH98}, and SDPNAL+ \cite{YST2015}. 
SDP solvers based on primal-dual interior-point algorithms such as
SeDuMi, SDPA, SDPT3  have limitation in solving dense 
SDP relaxations where the size of the variable matrix is at most several thousands. 
As the size of SDP relaxations in the hierarchy of SDP relaxations
for POPs grows exponentially with a parameter called the relaxation order determined by the degree of POPs, it is impossible to solve  POPs with 20-30 variables
using the SDP solvers based on  primal-dual interior-point algorithms unless some special
features of POPs such as 
sparsity or symmetry are utilized. Recently announced SDPNAL+ \cite{YST2015}
employs a majorized semismooth Newton-CG augmented Lagrangian method and has illustrated its superior performance
of solving large-scale  SDPs.

SDPNAL+, however, tends to exhibit some numerical difficulties when
 handling degenerate SDP problems, in particular 
problems with
many equality constraints. In these cases, 
it often cannot solve the degenerate SDP problems accurately and the lower bounds computed
are 
even invalid. 
It is frequently observed  in numerical computation that
the SDP 
problems in the hierarchy of dense or sparse SDP relaxations generally become more 
degenerate as the relaxation order is
increased to obtain tight lower bounds for the optimal value of POPs.
As shown in \cite{KIM2016},
the degeneracy also increases as  the number of variables and the degree of constrained POPs 
become larger.  Moreover,
the SDP relaxation of a high-degree POP 
can be degenerate even with the first 
relaxation order in many cases.
When a POP from applications can be represented in several different formulations,
the degeneracy of each formulation may differ.
In particular, it was discussed  in \cite{ITO17} that
different formulations of equivalent conic relaxations of a combinatorial 
quadratic optimization problem (QOP) can significantly affect the degeneracy. It was also shown through
the numerical results 
 that SDPNAL+ worked efficiently on some conic relaxation formulations of a QOP but 
 its performance on some other formulations was not satisfactory because of the degeneracy.
 Because of the limitation of SDPNAL+, the BP method in \cite{KIM2013,ARIMA2017} 
 was specifically designed to handle potentially degenerate SDP relaxations problems.
The BP method was demonstrated to be robust against the degeneracy in \cite{KIM2016}, while
applying SDPNAL+ to such degenerate  cases often leads to invalid bounds
and slow convergence. 
Thus, it is essential to have a solver that can deal with the degeneracy of SDP relaxations for computing 
valid lower bounds of large-scale or high-degree POPs.

The robustness of \matBP\ is guaranteed by the results in \cite{ARIMA2017} where the lower bounds obtained by 
the BP method is shown to be always valid. In addition, \matBP\ can effectively handle degenerate DNN relaxations of 
large-scale and/or high-degree POPs. 
We show through numerical experiment that \matBP\ can efficiently and robustly compute 
valid lower bounds for the optimal values of large-scale sparse POPs with 
BBC constraints  in comparison to SDPNAL+.
The test instances whose valid lower bounds could be obtained successfully by \matBP\ in 2000 seconds include, 
for example, a degree 3 binary POP with complementarity constraints in 1444 variables  
and a degree 8 box constrained POP in 120 variables.
For these instances, SDPNAL+ did not provide a comparable bound  within 20000 seconds.

A distinctive feature of  our package \matBP\ is that  it  not only 
automatically generate DNN relaxations 
for a BBC constrained POP 
but also integrate their computations with the robust 
BP algorithm that is designed specifically for solving them.
Other available
software packages and numerical methods for POPs such as   GlotiPoly  \cite{GLOPTIPOLY2003}, SOSTOOLs \cite{SOSTOOLS}, SparsePOP \cite{WAKI2008}, BSOS \cite{TOH17}, and SBSOS~\cite{WEISSER17}
in fact first generate the SDP relaxations for the underlying POPs, and then rely on 
existing SDP solvers such as SDPT3 or SeDuMi to solve the resulting SDP problems.
As a result, their performance are heavily dependent on the SDP solver chosen.

This paper is organized as follows: 
In Section 2, we describe a simple COP~\eqref{eq:generalCOP}  to which our DNN relaxation 
of POP~\eqref{POP0} is reduced. We also briefly explain the accelerated proximal gradient method and the BP method for solving the COP. 
In Section 3, we first describe how to exploit the sparsity in the POP \eqref{POP0} and then derive a simple COP of 
the form~\eqref{eq:generalCOP} to serve as the sparse DNN relaxations of 
\eqref{POP0}. 
In Section 4, we present computational aspects of \matBP \ and issues related to its 
efficient implementation. Section 5 contains numerical results on various BBC constrained 
POPs. Finally, we conclude in Section 6.

\section{Preliminaries}

We describe a simple COP~\eqref{eq:generalCOP}  in Section 2.1, 
the accelerated proximal 
gradient method \cite{BECK2009} in Section 2.2 and 
the bisection and projection (BP) method \cite{KIM2013,ARIMA2017} in Section 2.3. 
These two methods are designed to solve COP~\eqref{eq:generalCOP} and implemented in \matBP. 
In Section 3, we will reduce a DNN relaxation of POP~\eqref{POP0} to 
the simple COP~\eqref{eq:generalCOP}. 

Let $\spaceV$ be a finite dimensional vector space endowed with an inner product $\inprod{\cdot}{\cdot}$ 
and its induced norm $\left\| \cdot \right\|$ such that $\left\| \X \right\| = \left(\inprod{\X}{\X} \right)^{1/2}$ 
for every $\X \in \spaceV$. 
Let $\coneK_1$ and $\coneK_2$ be closed convex cones in $\spaceV$ satisfying 
$(\coneK_1 \cap \coneK_2)^* = (\coneK_1)^* + (\coneK_2)^*$, where
$\coneK^* = \{ \Y \in \spaceV :  \inprod{\X}{\Y} \geq 0 \ \mbox{for all } \X \in \coneK\}$ denotes the 
dual cone of a cone $\coneK \subset \spaceV$. 
Let $\Real^n$ be the space of $n$-dimensional column vectors, 
$\Real^n_+$ the nonnegative orthant of $\Real^n$, 
$\SymMat^n$ the space of $n \times n$ symmetric matrices,  
$\SymMat^n_+$ the cone of $n \times n$ symmetric positive semidefinite matrices,  
and $\SymN^n $ the cone of $n \times n$ symmetric nonnegative matrices. 

\subsection{A simple conic optimization problem}

Let  $\Q^0 \in \spaceV$ and $\O \not= \H^0 \in \coneK_1^* + \coneK_2^*$.
We introduce the following conic optimization problem (COP):
\begin{equation}
  \eta^* = \min_{\Z} \{\inprod{\Q_0}{\Z} \mid \inprod{\H_0}{\Z} = 1, ~\Z \in \coneK_1 \cap \coneK_2\}. \label{eq:generalCOP}
\end{equation}
If we take $\spaceV=\SymMat^m$, $\coneK_1=\SymMat^m_+$, and $\coneK_2$ a 
polyhedral cone in $\SymMat^m$ 
for some $m$, respectively, then the problem \eqref{eq:generalCOP} represents a general SDP. If in addition 
$\coneK_2 \subset \SymN^m$, then it forms a DNN optimization problem.  
Let $\G(y_0) = \Q_0 - y_0 \H_0$. 
The dual of \eqref{eq:generalCOP} can be described as 
\begin{eqnarray}
  y_0^* & = &  \max_{y_0,\Y_2}\{y_0 \mid  \Q_0 - y_0 \H_0 - \Y_2 \in \coneK_1^*, \ \Y_2 \in \coneK_2^*\}
\label{eq:generalDualCOP0} \\
 & = & \max_{y_0}\{y_0 \mid \G(y_0) \in \coneK_1^* + \coneK_2^*\}.\label{eq:generalDualCOP}
\end{eqnarray}
As shown in \cite[Lemma 2.3]{ARIMA2014},
strong duality  holds for \eqref{eq:generalCOP} and \eqref{eq:generalDualCOP}, {\it i.e.}, $\eta^*=y_0^*$.

Since $\H_0 \in \coneK_1^* + \coneK_2^*$ 
in \eqref{eq:generalDualCOP}, we have the following inequality from \cite{KIM2013}: 
\begin{equation}
y_0 \leq y_0^* ~ \text{if and only if} ~ \G(y_0) \in \coneK_1^* + \coneK_2^*.
\label{eq:feasibility}
\end{equation}
Therefore, the approximate value of $y_0^*$ can be computed by using the bisection method 
if the feasibility of any given $y_0$, {\it i.e.}, whether 
$\G(y_0) \in \coneK_1^* + \coneK_2^*$, 
can be determined. 
The recently proposed bisection and projection (BP) method \cite{KIM2013} 
(BP Algorithm described in Section 2.3) 
provides precisely
the feasibility test for any given $y_0$  through a numerical algorithm (APG Algorithm described in Section 2.2)
that is 
based 
on the accelerated proximal gradient method, where we employed the version in \cite{BECK2009}
that is modified from \cite{Nesterov}.

\subsection{The accelerated proximal gradient algorithm for feasibility test} \label{sec:APGforFeas}

For an arbitrary fixed $y_0$, let $\G = \G(y_0)$ for simplicity of notation. 
We also use the notation $\Pi_{\coneK}(\Z)$ to denote the metric projection of $\Z \in V$ onto a closed 
convex cone $\coneK \subset V$. 
Then  the problem of  testing whether $\G \in \coneK_1^* + \coneK_2^*$ leads to the following problem:
\begin{align}
   f^*
   &=\min_{\Y}\Bigl\{\frac{1}{2}\|\G - \Y\|^2 ~\Big|~ ~ \Y \in \coneK_1^* + \coneK_2^*\Bigr\} 
\label{eq:regressionFISTA1cone} \\
   &=\min_{\Y_1,\Y_2}\Bigl\{\frac{1}{2}\|\G - (\Y_1 + \Y_2)\|^2 ~\Big|~ ~ \Y_1 \in \coneK_1^*, ~ \Y_2 \in \coneK_2^*\Bigr\}
\nonumber \\
   &=\min_{\Y_1}\Bigl\{f(\Y_1):=\frac{1}{2}\|\Pi_{\coneK_2}(\Y_1 - \G)\|^2 ~\Big|~ ~ \Y_1 \in \coneK_1^*\Bigr\}  .
\label{eq:regressionFISTA}
\end{align}
Here we note that 
\begin{eqnarray*}
\min_{\Y_2}\Bigl\{
\frac{1}{2}\|\G - (\Y_1 + \Y_2)\|^2 ~\Big|~ ~ \Y_2 \in \coneK_2^* \Bigr \} 
& = &
\frac{1}{2} \| \G-\Y_1 - \Pi_{\coneK^*_2}(\G - \Y_1) \|^2 \\
& = & \frac{1}{2} \| \Pi_{\coneK_2}( \Y_1 - \G) \|^2, 
\end{eqnarray*}
where the last equality above follows from Moreau's decomposition theorem \cite{MOREAU62,Combettes2013}.
Obviously, $f^*\geq 0$, and  $f^*=0$ if and only if $\G\in\coneK_1^*+\coneK_2^*$. 

The gradient of the objective function $f(\Y_1)$ of \eqref{eq:regressionFISTA} is  given by
$\nabla f(\Y_1) = \Pi_{\coneK_2}(\Y_1 - \G)$. 
As the projection operator $\Pi_{\coneK}$ onto a convex set $\coneK$ is nonexpansive \cite[Proposition 2.2.1]{Bertsekas2003Convex}, we have that
\begin{equation*}
  \|\nabla f(\Y_1) - \nabla f(\Y_1')\| \leq \|(\Y_1 - \G) - (\Y_1' - \G)\| = \|\Y_1 - \Y_1'\| \quad (\Y_1,\Y_1'\in \spaceV).
\end{equation*}
Therefore, the gradient $\nabla f(\Y_1)$ is Lipschitz continuous with the Lipschitz constant $\Lip=~1$.
The KKT condition for $(\Y_1,\Y_2)$ to be the optimal solution of 
\eqref{eq:regressionFISTA1cone} is given by 
\begin{gather*}
  \X = \G - \Y_1 - \Y_2, \quad \inprod{\X}{\Y_1} = 0, \quad \inprod{\X}{\Y_2} = 0,\\
  \X\in\coneK_1\cap\coneK_2, \quad \Y_1\in \coneK_1^*, \quad \Y_2\in \coneK_2^*.
\end{gather*}
When the KKT condition above holds, we have that 
$\G \in \coneK_1^* + \coneK_2^*$ if and only if $\|\X\| = 0$.

Assume that the metric projections $\Pi_{\coneK_1}$ and $\Pi_{\coneK_2}$ onto the cones $\coneK_1$ and $\coneK_2$ can be computed
without difficulty.  
In \cite{KIM2013}, APG Algorithm \cite{BECK2009} below 
is applied to \eqref{eq:regressionFISTA} to determine whether  $f^*=0$  numerically.

Our APG based algorithm employs  the following error criterion
\begin{equation*}
  g(\X,\Y_1,\Y_2) = \max\Big\{
        \frac{ \inprod{\X}{\Y_1} }{ 1+\|\X\|+\|\Y_1\| }, 
        ~\frac{ \inprod{\X}{\Y_2} }{ 1+\|\X\|+\|\Y_2\| }, 
        ~\frac{ \Pi_{\coneK_1^*}(-\X) }{ 1+\|\X\| }, 
        ~\frac{ \Pi_{\coneK_2^*}(-\X) }{ 1+\|\X\| }
        \Big\}
\end{equation*}
to check whether $(\Y_1,\Y_2) = (\Y_1^k,\Y_2^k) \in \coneK_1^* \times \coneK_2^*$ 
satisfies the KKT conditions approximately. 
It terminates if $\|\X^k\|<\epsilon$ or if $\|\X^k\| \geq \epsilon$ and $g(\X^k,\Y_1^k,\Y_2^k) < \delta$ for sufficiently small positive 
$\epsilon$ and $\delta$, say $\epsilon = 10^{-12}$ and $\delta = 10^{-12}$. 
Note that $f(\Y_1^k)$ corresponds to $\frac{1}{2}\|\X^k\|^2$. Hence if $\|\X^k\|$ becomes smaller 
than $\epsilon > 0$, 
we may regard $(\Y_1^k,\Y_2^k) \in \coneK_1^* \times \coneK_2^*$ is an approximate 
optimal solution of \eqref{eq:regressionFISTA1cone} and $\G = \G(y_0) \in(\coneK_1^*+\coneK_2^*)$, 
which implies that $y_0$ is a feasible solution of the problem \eqref{eq:generalDualCOP} and 
$y_0 \leq y_0^*$. 
On the other hand, if $\|\X^k\| \geq \epsilon$ and $g(\X^k,\Y_1^k,\Y_2^k) < \delta$, then the KKT optimality condition is 
almost satisfied, and we classify that $\G=\G(y_0)$ does not lie 
in $\coneK_1^*+\coneK_2^*$ according to
$\|\X^k\| \geq \epsilon$. In the latter case, we determine that $y_0$ is not a 
feasible solution of \eqref{eq:generalDualCOP} and $y_0 > y_	0^*$.

\medskip

\noindent
\rule[0mm]{162mm}{0.5mm}\\
{\bf APG Algorithm} (the accelerated proximal gradient algorithm \cite{BECK2009} for feasibility test )\\
\rule[2mm]{162mm}{0.2mm}\vspace{-3mm}
    \begin{algorithmic}
    \STATE Input: $\G \in \spaceV$, $\Y_1^0 \in \spaceV$, $\Pi_{\coneK_1}$, $\Pi_{\coneK_2}$, $\epsilon>0$, $\delta>0$, $k_{max}>0$,  
    \STATE Output: $(\X,\Y_1,\Y_2) = (\X^{k},\Y_1^{k},\Y_2^{k})$
    \STATE Initialize: $t_1 \leftarrow 1, ~L \leftarrow  L_f(=1), ~\overline{\Y}_1^{1} \leftarrow \Y_1^0$ 
    \FOR{$k = 1,\dots,k_{max}$}
      \STATE $\Y_1^k \leftarrow \Pi_{\coneK_1^*}\Big(\overline{\Y}_1^k - \frac{1}{L}\Pi_{\coneK_2}(\overline{\Y}_1^k - \G)\Big)$  
      \STATE $\Y_2^{k+1} \leftarrow \Pi_{\coneK_2^*}(\G - \Y_1^{k})$, ~ $\X^{k} \leftarrow \G - \Y_1^{k} - \Y_2^{k}$
      \IF{ $\|\X^{k}\|<\epsilon$  (implying  $\G \in \coneK_1^*+\coneK_2^*$) or 
        {\bf if}  $\|\X^{k}\| \geq \epsilon$ and $g(\X^{k},\Y_1^{k},\Y_2^{k}) < \delta$
        (implying  $\G \not \in \coneK_1^*+\coneK_2^*$) 
	}
        \STATE \textbf{break}
      \ENDIF
      \STATE $t_{k+1}  \leftarrow  \frac{1+\sqrt{1+4t_k^2}}{2}$  
      \STATE $\overline{\Y}_1^{k+1} \leftarrow \Y_1^{k} + \frac{(t_k) - 1}{t_{k+1}} (\Y_1^{k} - \Y_1^{k-1})$  
    \ENDFOR
    \end{algorithmic}
\rule[4mm]{162mm}{0.5mm}\\

We note that 
the sublinear convergence of APG Algorithm, in the sense that $f(\Y_1^k) - f^* \leq O(1/k^2)$, 
is ensured for any optimal solution $\Y_1^*$ of \eqref{eq:regressionFISTA} in \cite[Theorem 4.4]{BECK2009}. 

\subsection{The bisection and projection algorithm for COP} \label{sec:BP}

As numerically small numbers $\epsilon >0$ and $\delta > 0$ must be used in APG Algorithm 
to decide whether $\|X^k\|$ is equal to  $0$
on a finite precision floating-point arithmetic machine, 
an infeasible $y_0$ can sometimes be erroneously
determined as a feasible solution of the problem \eqref{eq:generalDualCOP}  by 
APG Algorithm. Likewise, there is also a small possibility for a feasible solution to
be wrongly declared as infeasible due to numerical error.
As a result, the feasibility test based on APG Algorithm may not be always correct.

To address the validity issue of the result obtained from APG Algorithm, 
an improved BP method was introduced in  \cite{ARIMA2017}.  
Here, a valid lower bound $y_0^{v\ell}$ for the optimal value $y_0^*$ is always generated by the improved BP method,
which assumes the following two conditions  for a given interior point $\I$ of $\coneK_1^*$:
\begin{description}
  \item[(A1)] There exists a known positive number $\rho>0$ such that $\inprod{\I}{\Z}\leq\rho$ for every feasible solution $\Z$ 
of \eqref{eq:generalCOP}. 
  \item[(A2)]  For each $\Z\in\spaceV$, $\lambda_{\min}(\Z) = \sup\{\lambda \mid \Z - \lambda \I \in \coneK_1^*\}$ can be computed accurately at a moderate cost.
\end{description}
If $\coneK_1 = \coneK_1^*= \SymMat_+^n$ and $\I$ is the identity matrix, 
then $\inprod{\I}{\Z}$ is the trace of $\Z$ and $\lambda_{\min}(\Z)$ is the minimum eigenvalue of $\Z$. 
Under the  assumption (A1), the problem \eqref{eq:generalCOP} is equivalent to 
\begin{equation}
  \min_{\Z} \Big\{ \Inprod{\Q_0}{\Z} ~\Big|~ \inprod{\H_0}{\Z} = 1, ~\inprod{\I}{\Z} \leq \rho, ~\Z \in \coneK_1 \cap \coneK_2 \Big\}. \label{eq:generalCOPIX}
\end{equation}
Its dual 
\begin{equation}
  \sup_{y_0,\Y_2,\mu}\{y_0 + \rho\mu \mid \G(y_0) - \Y_2 - \mu\I \in \coneK_1^*, ~ \Y_2\in\coneK_2^*, ~\mu\leq0\}\label{eq:generalDualCOPIX}
\end{equation}
is equivalent to \eqref{eq:generalDualCOP} with the optimal value $y_0^*$. 
Suppose that $\bar{y}_0\in\Real$ and $\bar{\Y}_2 \in \coneK_2^*$ are given. 
Let $\bar{\mu}=\min\{0, \lambda_{\min}(\G(\bar{y}_0) - \bar{\Y}_2)\}$. 
Then $(y_0,\Y_2,\mu) = (\bar{y}_0,\bar{\Y}_2,\bar{\mu})$ is a feasible solution 
of \eqref{eq:generalDualCOPIX}, and $y_0^{v\ell} = \bar{y}_0 + \rho\bar{\mu}$ provides a valid lower bound for $y_0^*$. 
If $(\bar{y}_0,\bar{\Y}_2)$ is a feasible solution of \eqref{eq:generalDualCOP0}, 
then $y_0^{v\ell}=\bar{y}_0$ from $\bar{\mu} = 0$. 
The improved BP method in  \cite{ARIMA2017} is described in BP Algorithm below. 

\medskip

\noindent
\rule[0mm]{162mm}{0.5mm}\\
{\bf BP Algorithm} (The improved bisection-projection algorithm \cite{ARIMA2017})\\
\rule[2mm]{162mm}{0.2mm}\vspace{-3mm}
\begin{algorithmic}
  \STATE Input: $y_0^\ell\leq y_0^u,  ~tol>0, \rho>0$, \quad $\epsilon>0, \delta>0, \eta_r, \Pi_{\coneK_1}, \Pi_{\coneK_2}, k_{\max}$.
  \STATE Output: $y_0^{v\ell}$
  \STATE Initialize: $y_0^m \leftarrow \frac{y_0^\ell + y_0^u}{2}$ if $-\infty <y_0^\ell$; otherwise, $y_0^m \leftarrow   y_0^u$. 
$\hat{\Y}_1 \leftarrow \Pi_{\coneK_1^*}\big(\G(y_0^m)\big)$. $y_0^{v\ell}  \leftarrow  -\infty$.
  \WHILE{$y_0^u - y_0^\ell > tol$}
    \STATE $\hat{\Y}_1^\mathrm{init} \leftarrow \hat{\Y}_1$
    \STATE $(\hat{\X},\hat{\Y_1},\hat{\Y_2}) \leftarrow $ The output of APGR Algorithm with inputs $\scriptstyle\big(\G(y_0^m),\hat{\Y}_1^{\mathrm{init}},\Pi_{\coneK_1},\Pi_{\coneK_2},\epsilon,\delta,k_{\max},
\eta_r\big)$.
    \STATE $y_0^{v\ell} \leftarrow \max\big\{y_0^{v\ell}, y_0^m+\rho\min\{0,\lambda_{\min}(\G(y_0^m) - \hat{\Y}_2)\}\big\}$ 
    \IF{$\|\hat{\X}\| < \epsilon$}
      \STATE $y_0^\ell \leftarrow y_0^m$ 
    \ELSE
      \STATE $y_0^u \leftarrow y_0^m$, ~~~ $y_0^\ell \leftarrow \max\{y_0^\ell, y^{v\ell}_0\}$
    \ENDIF
    \STATE $y_0^m \leftarrow (y_0^{\ell} + y_0^{u}) / 2$
  \ENDWHILE
\end{algorithmic}
\rule[4mm]{162mm}{0.5mm}\\

Here APGR Algorithm is an enhanced version of APG Algorithm which is described in Section 4.2.
One of the advantages of BP Algorithm above is that it does not require an initial finite estimation of the lower bound
 $y_0^{\ell}$. 
In fact, BP Algorithm implemented in  the current version of \matBP \ sets $y_0^m\leftarrow y_0^{u}$ 
and $y_0^\ell\leftarrow -\infty$ at the initialization step.
For a better upper bound $y_0^u$ for $y_0^*$, a heuristic method applied to the original POP can be employed. 
For the value of $\rho$, the exact theoretical value of $\rho$  can be computed 
from COP~\eqref{eq:generalCOP} in some cases; see \cite{ARIMA2014b} for example.
We also present a method to estimate $\rho$ in Section~\ref{sec:SubmoForRho}.

\section{Sparse DNN relaxation of POP~\eqref{POP0}}

\label{sec:COPandPOP}

Exploiting sparsity is a key technique for solving large-scale SDPs and 
SDP relaxations of large-scale POPs. 
For example, see~\cite{FUKUDA2003,WAKI2006,WAKI2008}. 
Throughout this section, we assume that POP~\eqref{POP0} satisfies a certain structured sparsity
described in Section 3.2, and derive 
its sparse DNN relaxation~\eqref{eq:relaxOfPOP} of the same form as 
COP~\eqref{eq:generalCOP}. 
The method to derive the DNN relaxation~\eqref{eq:relaxOfPOP} from POP~\eqref{POP0} 
can be divided into two main steps: lift POP~\eqref{POP0} to an 
equivalent problem~\eqref{eq:BBCPOPMoment} with moment matrices 
in the vector variable $\x$, and replace the moment matrices by symmetric matrix variables after 
adding valid inequalities. In the software \matBP,
the function \texttt{BBCPOPtoDNN} implements the two steps just mentioned. 
After introducing notation and symbols in Section~\ref{sec:notation}, 
we present how  the sparsity in POP~\eqref{POP0} is exploited in Section 3.2, and the 
details of the above
two steps in Section \ref{sec:POPMoment} and Section \ref{sec:ValidConstr}, respectively.

\subsection{Notation and symbols}

\label{sec:notation}

Let $\Integer^n_+$ denote the set of $n$-dimensional nonnegative integer vectors. 
For each $\x = (x_1,\ldots,x_n) \in \Real^n$ and $\balpha = (\alpha_1,\ldots,\alpha_n) \in \Integer^n_+$, let 
$\x^{\salpha} = x^{\salpha_1} \cdots x^{\salpha_n}$ denote a {\em monomial}. 
We call deg$(\x^{\salpha}) = \max\{ \alpha_i : i=1,\ldots,n \}$ the {\em degree} of a monomial $\x^{\salpha}$. 
Each polynomial $f(\x)$ is represented as $f(\x) = \sum_{\salpha \in \sFC} c_{\balpha} \x^{\salpha}$
for some nonempty finite subset $\FC$ of $\Integer^n_+$ and $c_{\balpha} \in \Real$ $(\balpha \in \FC)$. 
We call supp$f = \{ \balpha \in \FC : c(\balpha) \not= 0 \}$ the {\em support} of $f(\x)$; hence 
 $f(\x) = \sum_{\salpha \in \mbox{\small supp$f$}} c(\balpha) \x^{\salpha}$ is the minimal 
 representation of $f(\x)$. We call deg$f = \max\{ \mbox{deg} (\x^{\salpha}) : \balpha \in \mbox{supp}f \}$ 
 the {\em degree} of $f(\x)$.

Let $\AC$ be a nonempty finite subset of $\Integer_+^n$ with 
cardinality $\left| \AC \right|$, and 
let $\SymMat^{\sAC}$ denote the linear space of $\left| \AC \right| \times \left| \AC \right|$ symmetric matrices 
whose rows and columns are indexed by $\AC$. The $(\balpha,\bbeta)$th 
component of  each $\X \in \SymMat^{\sAC}$ is written as $X_{{\salpha}{\sbeta}}$ 
$((\balpha , \bbeta) \in \AC \times \AC)$. The {\em inner product} of $\X, \ \Y 
\in \SymMat^{\sAC}$ is defined by 
$\inprod{\X}{\Y} = \sum_{\salpha\in \sAC} \sum_{\sbeta \in \sAC} X_{\salpha\sbeta}Y_{\salpha\sbeta}$, and 
the {\em norm} of $\X \in \SymMat^{\sAC}$ is defined
 by $\left\| \X \right\|  = \left( \inprod{\X}{\X} \right)^{1/2}$. 
 Assuming that the elements of 
$\AC$ are enumerated in an appropriate order,
we denote a $\left| \AC \right|$-dimensional column 
 vector of monomials $\x^{\salpha}$ $(\balpha \in \AC)$ by 
$\x^{\sAC}$, and a $\left| \AC \right| \times \left| \AC \right|$ symmetric matrix $(\x^{\sAC})( \x^{\sAC})^T$ of monomials 
$\x^{\salpha+\sbeta}$ $((\balpha , \bbeta) \in \AC \times \AC)$ by $\x^{\sAC\times\sAC} \in \SymMat^{\sAC}$. 
We call $\x^{\sAC\times\sAC}$ a {\em moment matrix}. 

For a pair of subsets $\AC$ and $\BC$ of $\Integer^n_+$, let 
$\AC + \BC = \{ \balpha+\bbeta : \balpha \in \AC, \ \bbeta \in \BC \}$ denote their Minkowski sum. 
Let $\SymMat^{\sAC}_+$ denote the cone of positive semidefinite matrices in $\SymMat^{\sAC}$, 
and $\SymN^{\sAC}$ the cone of nonnegative matrices in $\SymMat^{\sAC}$. 
By construction, $\x^{\sAC\times\sAC} \in \SymMat^{\sAC}_+$ for every $\x \in \Real^n$, and 
$\x^{\sAC\times\sAC} \in \SymMat^{\sAC}_+ \cap \SymN^{\sAC}$ for every $\x \in \Real^n_+$. 

We denote the feasible region of POP \eqref{POP0} as 
\begin{eqnarray*}
  \BBC = \{\x \in \Real^n \mid x_i\in[0,1] ~ (i\in\Ibox), \quad x_j\in\{0,1\} ~ (j\in\Ibinary), \quad \x^{\bgamma}=0 ~ (\bgamma\in\Gamma)\},
\end{eqnarray*}
where $\Gamma = \bigcup_{C \in \CC} \{ \bgamma \in \{0,1\} ^n \mid  \gamma_i = 1 ~ \text{if} ~ i \in C ~ \text{and} ~ \gamma_j = 0 ~ \text{otherwise} \}$. 
Then, POP \eqref{POP0} is written as follows:
\begin{equation}
  \zeta^* = \min_{\x\in\Real^n} \{ f_0(\x) \mid \x \in \BBC \}. \label{eq:BBCPOP}
\end{equation}
We note that the POPs dealt with in \cite{KIM2016} are special cases
of \eqref{eq:BBCPOP} where $\Gamma=\emptyset$. 

Let  $\r:\Integer_+^n\to \Integer_+^n$ be defined by
\begin{equation}
  (\r(\balpha))_i = \begin{cases}
    \min\{\alpha_i, 1\} & \text{if } i\in\Ibinary \\
    \alpha_i & \text{otherwise ({\it i.e.}, $i\in\Ibox$).}
  \end{cases}\label{eq:r_function}
\end{equation}
If $\x\in \BBC$, then $\x^{\balpha}=\x^{\r(\balpha)}$ holds for all $\balpha \in \Integer_+^n$.
Hence we may replace each monomial $\x^{\balpha}$ in $f_0(\x)$ by $\x^{\r(\balpha)}$. Therefore 
$\supp f_0 = \r(\supp f_0)$ is assumed without loss of generality in the subsequent discussion.

\subsection{Exploiting sparsity}

Let $\nabla^2f_0(\x)$ denote the Hessian matrix of $f_0(\x)$. 
For POP~\eqref{eq:BBCPOP}, we introduce 
{\it the $n \times n$ sparsity pattern matrix} $\R$ whose $(i,j)$th element is defined by 
\begin{eqnarray*}
R_{ij} & = & \left\{
\begin{array}{ll}
1  &  \mbox{if }  i=j \ \mbox{or  the } (i,j)\mbox{th element of } \nabla^2f_0(\x) \ \mbox{is not identically zero}, \\
1 & \mbox{if } i,j\in C \ \mbox{for some } C \in \CC, \\ 
0 & \mbox{otherwise}. 
\end{array}
\right. 
\end{eqnarray*}
If $\CC = \emptyset$, then $\R$ represents the sparsity pattern of the Hessian matrix $\nabla^2f_0(\x)$. 

Next, we choose a family of subsets $V^k$ of $N=\{1,2,\ldots,n\}$ $(k=1,\ldots,\ell)$ 
such that the union of $V^k \times V^k$ 
$(k=1,\ldots,\ell)$ covers the set of indices $(i,j)$ associated with the nonzero elements of $\R$, i.e.,
\begin{equation}
\big\{ (i,j) \in N \times N  \mid  R_{ij} = 1 \big\} \subseteq \bigcup_{k=1}^{\ell}  V^k \times V^k. \label{eq:Vk}
\end{equation} 
Obviously, 
when the only set
 $V^1=N$ is chosen for such a family, 
 we get a dense DNN relaxation of 
POP~\eqref{eq:BBCPOP}. When $\R$ is sparse (see 
Figure~\ref{fig:spy}),  {\it the sparsity pattern graph} $G(N,\EC)$ with the node 
set $N$ and the edge set $\EC = \{(i,j) \in N \times N : i < j \ \mbox{and } R_{ij} = 1\}$ is utilized to 
create such a family $V^k$ $(k=1,\ldots,\ell)$. More precisely, let $G(N,\overline{\EC})$ be a chordal 
extension of $G(N,\EC)$, and take the maximal cliques $V^k$ $(k=1,\ldots,\ell)$ for the family. 
The chordal extension and its maximal cliques can be found by using the technique \cite{BLAIR1993} 
based on the symbolic Cholesky decomposition of the adjacency matrix of $G(N,\EC)$. 
We implicitly assume that POP~\eqref{eq:BBCPOP} satisfies the structured sparsity which induces small size 
$V^k$ with $\left|V^k\right| = O(1)$  $(k=1,\ldots,\ell)$. Figure~\ref{fig:spy} 
shows such examples. The family $V^k$ $(k=1,\ldots,\ell)$ chosen this way satisfies nice properties;
see \cite{BLAIR1993,FUKUDA2003,WAKI2006} for more details. 
Here we only mention that the number $\ell$ of maximal cliques does not exceed $n$.

As representative sparsity patterns, we illustrate the following two types: 
\begin{description}
  \item[Arrow type:] For given $\ell\geq2$, $a\geq2$, $b \in \{0,\ldots,a-1\}$, and $c\geq 1$, we set
  \begin{equation*}
    V^k = (\{(k-1)(a-b)\}+\{1, 2, \ldots a\}) \cup (\{(\ell-1)(a-b)+a\} + \{1, 2, \ldots, c\}) \ (k=1,\ldots,\ell)
  \end{equation*}
(the left picture of Figure~\ref{fig:spy}).
  \item[Chordal graph type:] Let the number $n$ of variables and the \textit{radio range} $\rho >0$ be given. 
  For $n$ points $\v_1,\v_2,\ldots,\v_n$ drawn from  a uniform distribution over the unit square $[0,1]^2$, 
  we construct the sparsity pattern graph $G(N,\EC)$ such that 
$\EC=\{(i,j) \in N\times N \mid i < j, \ \|\v_i - \v_j\|\leq \rho \}$, 
where $N=\{1,\ldots,n\}$.
  Let $V^k$ $(k=1,2,\ldots,\ell)$ be the maximal cliques in a chordal extension of $G=(N,\EC)$. 
(the right picture of Figure~\ref{fig:spy}). 
\end{description}
In Section 5, we report numerical results on randomly generated instances of binary and box constrained 
POPs with these two  types of sparsity patterns.

\begin{figure}
  \centering
  \begin{minipage}{.32\textwidth}
    \centering
    \includegraphics[width=.99\textwidth]{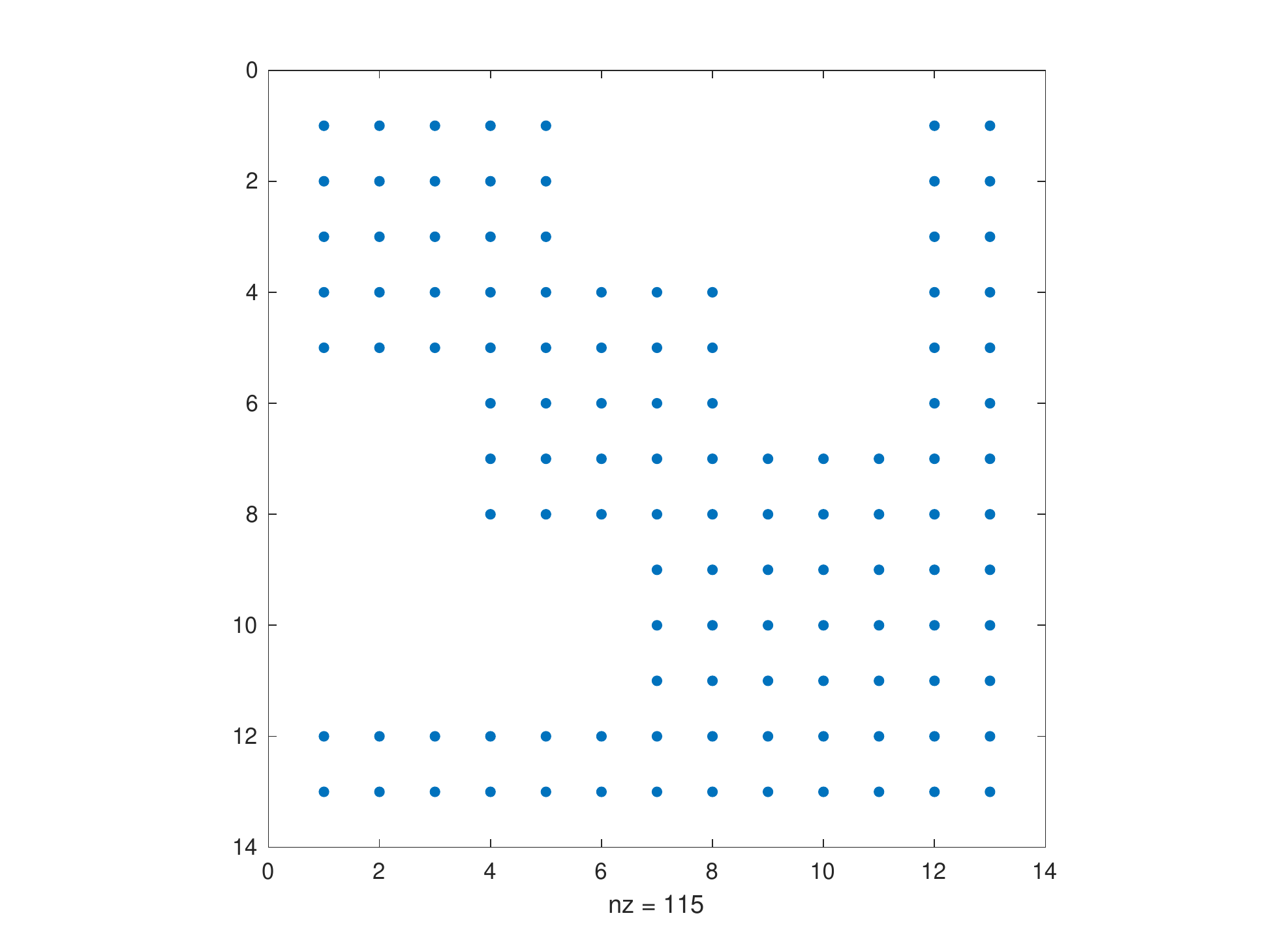}\\
    Arrow type ($(a,b,c,\ell)=(5, 2, 2, 3)$)
  \end{minipage}
  \begin{minipage}{.32\textwidth}
    \centering
    \includegraphics[width=.99\textwidth]{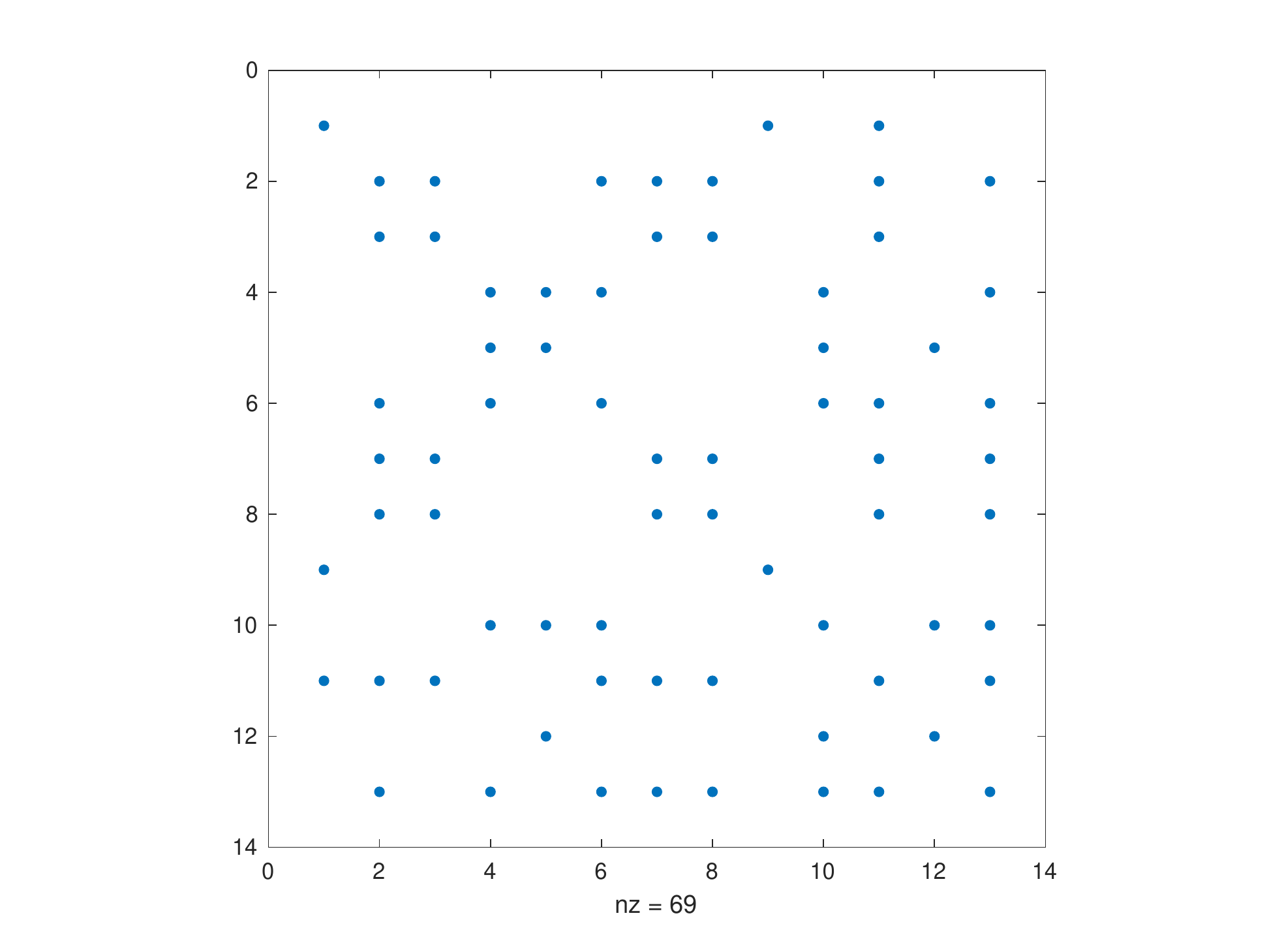}\\
    Chordal graph type \\($r=0.4$)
  \end{minipage}
  \caption{Examples of the sparsity pattern matrix $\R$ 
with $n=13$, where the dots correspond to $1$ and the blank parts to $0$'s.}\label{fig:spy}
\end{figure}

\subsection{Lifting POP~\eqref{eq:BBCPOP} with moment matrices in $\x$} \label{sec:POPMoment}

Let 
\begin{align*}
d & = \max\big\{\deg(f_0), ~ \deg(\x^{\bgamma})\ (\bgamma\in\Gamma) \big\}, \ 
\lceil{d/2}\rceil \leq \omega \in \Integer_+\\
\AC^k_\omega &=\big\{\balpha\in\Integer_+^n \mid \alpha_i=0 ~(i\not\in V^k), ~\sum_{i\in V^k}\alpha_i \leq \omega \big\} 
\ (k=1,\ldots,\ell). 
\end{align*}
Here the parameter $\omega$ is named as the \textit{relaxation order}. We then see that 
\begin{equation}
  \Gamma \subseteq \bigcup_{k=1}^\ell (\AC^k_\omega+\AC^k_\omega) \ \mbox{and } 
  \supp(f_0) \subseteq \bigcup_{k=1}^\ell (\AC^k_\omega+\AC^k_\omega). \label{eq:coverSupport}
\end{equation}
By the first inclusion relation in \eqref{eq:coverSupport}, 
each monomial $\x^{\bgamma}$ $(\bgamma \in \Gamma)$ is involved in the moment matrix $\x^{\AC^k_{\omega}\times\AC^k_{\omega}}$ for some $k \in \{1,\ldots,\ell\}$. By the second inclusion relation 
in \eqref{eq:coverSupport}, 
each monomial $\x^{\balpha}$ of the polynomial $f_0(\x)$ is involved in the moment matrix $\x^{\AC^k_{\omega}\times\AC^k_{\omega}}$ for some $k \in \{1,\ldots,\ell\}$. Hence the polynomial objective function $f_0(\x)$ can be lifted to the space 
$\spaceV = \SymMat^{\AC^1_{\omega}} \times \cdots \times \SymMat^{\AC^{\ell}_{\omega}} $ such that 
\begin{equation*}
f_0(\x) = \inprod{\F_0}{(\x^{\AC^1_{\omega}\times\AC^1_{\omega}},\ldots,\x^{\AC^{\ell}_{\omega}\times\AC^{\ell}_{\omega}})} = 
\sum_{k=1}^{\ell} \inprod{\F_0^k}{\x^{\AC^k_{\omega}\times\AC^k_{\omega}}}
\end{equation*}
for some $\F_0 = (\F_0^1,\ldots,\F_0^{\ell}) \in \spaceV$, 
where the inner product $\inprod{\A}{\B}$ for each pair of $\A=(\A^1,\ldots,\A^{\ell})$ and $\B=(\B^1,\ldots,\B^{\ell})$ in 
$\spaceV = \SymMat^{\AC_1}\times\cdots\times\SymMat^{\AC_\ell}$ is defined by 
$\inprod{\A}{\B}=\sum_{k=1}^\ell\inprod{\A^k}{\B^k}$. 
 To lift the constraint set $H\subseteq\Real^n$ 
to the space $\spaceV$, we define
\begin{align*}
M &= \{ (\x^{\AC^1_{\omega}\times\AC^1_{\omega}},\ldots,\x^{\AC^{\ell}_{\omega}\times\AC^{\ell}_{\omega}}) \in \spaceV \mid 
\x \in H \} \\ 
  &= \left\{ (\x^{\AC^1_{\omega}\times\AC^1_{\omega}},\ldots,\x^{\AC^{\ell}_{\omega}\times\AC^{\ell}_{\omega}}) \in \spaceV ~\middle|  
\begin{array}{l} 
x_i \in [0,1] \ (i \in \Ibox), \ x_i \in \{0,1\} \ (i \in \Ibinary), \\ 
\x^{\bgamma} = 0 \ (\bgamma \in \Gamma)
\end{array} 
\right\}. 
\end{align*}
By definition, $\x \in H$ if and only if 
$(\x^{\AC^1_{\omega}\times\AC^1_{\omega}},\ldots,\x^{\AC^{\ell}_{\omega}\times\AC^{\ell}_{\omega}}) \in M$. 
Thus we can lift POP  \eqref{eq:BBCPOP} to the space $\spaceV$ as follows:
\begin{equation}
  \min_{\Z} \left\{ 
  \inprod{\F_0}{\Z} ~\middle|~ \Z=(\Z^1,\Z^2,\ldots,\Z^\ell)\in \MomentBBC
  \right\}.  \label{eq:BBCPOPMoment}
\end{equation}

To illustrate the lifting procedure from POP  \eqref{eq:BBCPOP} to the space $\spaceV$, 
we consider the following example.

\begin{example}
{\rm 
\label{example:momentRepresentation}
Let us consider the following POP with $n=3$, 
$\CC=\{\{1,2\}\}$, 
$\Ibox=\{1\}$, and $\Ibinary=\{2,3\}$:
\begin{equation}
  \min_{\x\in\Real^3} 
  \left\{ 
    f_0(\x) = - x_1x_2 - x_2x_3
  ~\middle|~ 
  \begin{array}{l}
    x_1x_2 = 0, \quad x_1\in[0,1], \quad x_2,x_3\in\{0,1\}.
  \end{array}
  \right\}.\label{eq:examplePOP}
\end{equation}
Since $d = \max\{\deg(f_0), ~ \deg(\x^{\bgamma})\ (\bgamma\in\Gamma) \} = 2$, we can take $\omega=1 \geq \lceil d/2 \rceil$. 
The sparsity pattern matrix $\R$ turns out to be 
$ 
\R =  \left(\begin{smallmatrix} 1 & 1 & 0 \\ 1 & 1 & 1 \\ 0 & 1 & 1 \end{smallmatrix}\right). 
$ 
We show two different choices of $V^k$ $(k=1,\ldots,\ell)$. 

\paragraph{Dense case:}
Let $\ell=1$, $V^k=\{1,2,3\}$, and $\omega=1$. 
Then
$
\AC^1_\omega=
\left\{
\left(\begin{smallmatrix}0\\0\\0\end{smallmatrix}\right),
\left(\begin{smallmatrix}1\\0\\0\end{smallmatrix}\right),
\left(\begin{smallmatrix}0\\1\\0\end{smallmatrix}\right),
\left(\begin{smallmatrix}0\\0\\1\end{smallmatrix}\right)
\right\}
$.
We have
\[\x^{\AC^1_\omega}=\begin{pmatrix}1\\x_1\\x_2\\x_3\end{pmatrix} \quad \text{and} \quad
\x^{\AC^1_\omega\times\AC^1_\omega}=\x^{\AC^1_\omega}(\x^{\AC^1_\omega})^T =
\begin{pmatrix}
  1 & x_1 & x_2 & x_3\\
  x_1 & x_1^2 & x_1x_2 & x_1x_3\\
  x_2 & x_1x_2 & x_2^2 & x_2x_3\\
  x_3 & x_1x_3 & x_2x_3 & x_3^2
\end{pmatrix}.
\]
If we define 
\[
\F_0^1=\begin{pmatrix}
  0 & 0 & 0 & 0\\
  0 & 0 & -0.5 & -0.5\\
  0 & -0.5 & 0 & 0\\
  0 & -0.5 & 0 & 0
\end{pmatrix}, 
\]
then $f_0(\x) = \inprod{\F_0^1}{\x^{\AC^1_\omega\times\AC^1_\omega}}$ holds. 

\paragraph{Sparse case:}
Let $\ell=2$, $V^1=\{1,2\}$, $V^2=\{2,3\}$, and $\omega=1$. 
Then 
$
\AC^1_\omega=\left\{
\left(\begin{smallmatrix}0\\0\\0\end{smallmatrix}\right),
\left(\begin{smallmatrix}1\\0\\0\end{smallmatrix}\right),
\left(\begin{smallmatrix}0\\1\\0\end{smallmatrix}\right)
\right\}
$ and 
$
\AC^2_\omega=\left\{
\left(\begin{smallmatrix}0\\0\\0\end{smallmatrix}\right),
\left(\begin{smallmatrix}0\\1\\0\end{smallmatrix}\right),
\left(\begin{smallmatrix}0\\0\\1\end{smallmatrix}\right)
\right\}
$.
We have
\begin{align*}
  &\x^{\AC^1_\omega} = \begin{pmatrix}1\\x_1\\x_2\end{pmatrix}, 
  &&\x^{\AC^2_\omega} = \begin{pmatrix}1\\x_2\\x_3\end{pmatrix},\\
  &\x^{\AC^1_\omega\times\AC^1_\omega} = 
  \begin{pmatrix}
  1 & x_1 & x_2\\
  x_1 & x_1^2 & x_1x_2\\
  x_2 & x_1x_2 & x_2^2
  \end{pmatrix},
  &&\x^{\AC^2_\omega\times\AC^2_\omega} = 
  \begin{pmatrix}
  1 & x_2 & x_3\\
  x_2 & x_2^2 & x_2x_3\\
  x_3 & x_2x_3 & x_3^2
  \end{pmatrix}.
\end{align*}
If we define 
\begin{align*}
  &\F_0= (\F_0^1, \F_0^2) = 
  \begin{pmatrix}
    \begin{pmatrix}
      0 & 0 & 0\\
      0 & 0 & -0.5\\
      0 & -0.5 & 0
    \end{pmatrix},~
    \begin{pmatrix}
      0 & 0 & 0\\
      0 & 0 & -0.5\\
      0 & -0.5 & 0
    \end{pmatrix}
  \end{pmatrix},
\end{align*}
then $f_0(\x) = \sum_{k=1}^\ell \inprod{\F_0^k}{\x^{\AC^k_\omega\times\AC^k_\omega}}$ holds.
}
\end{example}

\subsection{Valid constraints and conic relaxations of POPs} \label{sec:ValidConstr}

Note that the objective function of the lifted minimization problem~\eqref{eq:BBCPOPMoment} is linear 
so that it is equivalent to the minimization of the same objective function over the convex hull of the feasible region 
$M$ of~\eqref{eq:BBCPOPMoment}.   However, the resulting convex minimization problem as well as the 
original problem~\eqref{eq:BBCPOPMoment} are numerically intractable. 
The second step for deriving a DNN relaxation from the POP \eqref{eq:BBCPOP} is to relax
the nonconvex feasible region $\MomentBBC$ to a numerically tractable convex set, 
which is represented as the intersection of the hyperplane of the form 
$\{ \Z \in \spaceV \mid \inprod{\H_0}{\Z} = 1\}$ 
and two convex cones $\coneK_1$ and $\coneK_2$ in $\spaceV$. Hence we  obtain a COP 
of the form~\eqref{eq:generalCOP}, which serves a DNN relaxation of POP~\eqref{eq:BBCPOP}. 

By definition, $(\0,\0) \in \AC^k_{\omega} \times \AC^k_{\omega}$ $(k=1,\ldots,\ell)$, which implies that
$(\x^{\AC^k_\omega \times \AC^k_\omega})_{\0\0} = 1$. Thus,  if $\H_0 = (\H_0^1,\ldots,\H_0^\ell)\in \spaceV$ 
is defined such that 
\begin{equation*}
  (\H_0^k)_{\balpha\bbeta} = \begin{cases}
    1/\ell  & \text{if } \balpha = \bbeta = \0\\
    0 & \text{otherwise,}
  \end{cases}
\end{equation*}
then the hyperplane $\{ \Z \in \spaceV \mid \inprod{\H_0}{\Z} = 1\}$  will contain $M$.
 We also know 
$M \subseteq \SymMat^{\AC^1_\omega}_+\times\cdots\times\SymMat^{\AC^\ell_\omega}_+$. As a result,
we can take 
$\coneK_1 = \SymMat^{\AC^1_\omega}_+\times\cdots\times\SymMat^{\AC^\ell_\omega}_+$. 

To construct the polyhedral cone $\coneK_2$, we consider the following valid equalities and inequalities 
for $M$: 
\begin{align}
& (\x^{\AC^k_\omega \times \AC^k_\omega})_{\balpha\bbeta}  \geq 0, \nonumber \\ 
&(\x^{\AC^k_\omega \times \AC^k_\omega})_{\balpha\bbeta} = (\x^{\AC^{k'}_\omega \times \AC^{k'}_\omega})_{\balpha'\bbeta'} & &
\mbox{if } 
\r(\balpha+\bbeta) = \r(\balpha'+\bbeta'), \nonumber \\
& (\x^{\AC^k_\omega \times \AC^k_\omega})_{\balpha\bbeta} \geq (\x^{\AC^{k'}_\omega \times \AC^{k'}_\omega})_{\balpha'\bbeta'} 
& &\mbox{if } 
\r(c(\balpha+\bbeta))=\r(\balpha'+\bbeta') \mbox{ for some } \ c \geq 1,  \label{eq:box} \\
&(\x^{\AC^k_\omega \times \AC^k_\omega})_{\balpha\bbeta} = 0 & &  \mbox{if } \r(\balpha+\bbeta)\geq\bgamma \ \mbox{for some } 
 \bgamma \in\Gamma, \nonumber 
\end{align}
$(k, k' \in \{1,\ldots,\ell\}, \ \balpha,\bbeta \in \AC^k_{\omega}, \
\balpha',\bbeta' \in \AC^{k'}_{\omega})$. 
In the above, the first inequality follows from the fact that $\x \geq \0$, the second and third inequalities follow
 from the definition of 
$\r : \Integer^n_+ \rightarrow \Integer^n_+$ and $\x \in [0,1]^n$, and the last equality from the complementarity 
condition $\x^{\bgamma} = 0$ $(\bgamma \in \Gamma)$.  Now, by linearizing the equalities and 
inequalities above, i.e., replacing $(\x^{\AC^1_\omega \times \AC^1_\omega},\ldots,
\x^{\AC^{\ell}_\omega \times \AC^{\ell}_\omega})$ by an independent 
variable 
$\Z = (\Z^1,\ldots,\Z^{\ell}) \in \spaceV$, we obtain the linear equalities and inequalities to describe 
the cone $\coneK_2$ such that 
\begin{equation*}
  \coneK_2 = 
  \left\{
      \Z=(\Z^1,\ldots,\Z^\ell) 
      \in \spaceV 
  ~\middle|~ 
    \begin{array}{l}
    \begin{array}{ll}
      \Z^k_{\balpha\bbeta} \geq 0 &  \mbox{(nonnegativity)} \\[3pt]
      \Z^k_{\balpha\bbeta} = \Z^{k'}_{\balpha'\bbeta'} & \mbox{if } \r(\balpha+\bbeta) = \r(\balpha'+\bbeta'),\\[3pt]
      \Z^k_{\balpha\bbeta} \geq \Z^{k'}_{\balpha'\bbeta'} &  
           \mbox{if } \r(c(\balpha+\bbeta)) = \r(\balpha'+\bbeta') \\
         & \mbox{for some } c \geq 1, \\[3pt]
      \Z^k_{\balpha\bbeta} = 0 & \mbox{if } 
          \r(\balpha+\bbeta)\geq\bgamma \ \mbox{for some $\gamma \in \Gamma$} \\[3pt]
   \end{array}\\
      (k, k' \in \{1,\ldots,\ell\}, \ \balpha,\bbeta \in \AC^k_{\omega}, \ \balpha',\bbeta' \in \AC^{k'}_{\omega})
    \end{array}
  \right\}.
\end{equation*}
Consequently, we obtain the following COP:
\begin{equation}
  \zeta = \min_{\Z} \left\{ 
  \inprod{\F_0}{\Z}
  ~\middle|~ 
    \inprod{\H_0}{\Z} = 1, ~~ \Z \in \coneK_1 \cap \coneK_2
  \right\}, \label{eq:relaxOfPOP}
\end{equation}
which serves as a DNN relaxation of POP~\eqref{eq:BBCPOP}. 

Note that COP \eqref{eq:relaxOfPOP} is exactly in the form of \eqref{eq:generalCOP},
to which BP Algorithm can be applied. 
In APG Algorithm, which is called within  BP Algorithm, 
the metric projection $\Pi_{\coneK_1}(\Z)$ of $\Z=(\Z^1,\ldots,\Z^\ell) \in \spaceV$ onto 
$\coneK_1=\SymMat^{\AC^1_\omega}_+\times\cdots\SymMat^{\AC^\ell_\omega}_+$ 
can be computed by the eigenvalue decomposition of $\Z^k\in\SymMat^{\AC^k_\omega}$ $(k=1,\ldots,\ell)$, 
and the metric projection $\Pi_{\coneK_2}(\Z)$ of $\Z=(\Z^1,\ldots,\Z^\ell) \in \spaceV$ onto 
$\coneK_2$ can be computed efficiently by Algorithm 3.3 of \cite{KIM2016}. 
As a result, COP~\eqref{eq:relaxOfPOP} can be solved efficiently by BP Algorithm. 
Note that the eigenvalue decomposition of $\Z^k\in\SymMat^{\AC^k_\omega}$, which requires 
$O(\left|\AC^k_\omega\right|^3)$ arithmetic operations at each iteration of APG Algorithm, is the most 
time consuming part in BP Algorithm applied to COP~\eqref{eq:relaxOfPOP}. 
Therefore, it is crucial for the
computational efficiency of BP Algorithm to choose smaller $V^k$ $(k=1,\ldots,\ell)$, which 
determines the size $\left|\AC^k_\omega\right|$ of $\AC^k_\omega$ 
$(k=1,\ldots,\ell)$,  by 
exploiting the sparsity of POP~\eqref{eq:BBCPOP} as  presented in Section 3.2.
We also note that the primal-dual interior-point method for COP~\eqref{eq:relaxOfPOP} would remain 
computationally expensive 
since $\coneK_2$ consists of a large number of inequality such as nonnegative constraints. 

Given a POP of the form~\eqref{eq:BBCPOP}, the sparsity pattern matrix $\R$ 
and the graph $G(N,\EC)$ are uniquely determined.  However, the choice of the family $V^k$ $(k=1,\ldots,\ell)$ 
satisfying~\eqref{eq:Vk} is not unique. 
Recall that $\F_0$, $\H_0$, $\coneK_1$ and $\coneK_2$  in COP~\eqref{eq:relaxOfPOP} 
depends on the relaxation order $\omega$. Theoretically, the optimal value $\zeta$ of 
COP~\eqref{eq:relaxOfPOP} is monotonically nondecreasing with respect to $\omega \geq \lceil d/2 \rceil$.
Thus, a tighter lower bound for the optimal value $\zeta^*$ of POP~\eqref{eq:BBCPOP} can be expected 
when a larger  $\omega$ is used. However, the
numerical cost of solving COP~\eqref{eq:relaxOfPOP} by BP Algorithm increases very rapidly 
as $\omega$ increases. 
We should mention that
the family of COP \eqref{eq:relaxOfPOP} with increasing  $\omega \geq \lceil d/2 \rceil$ 
forms a hierarchy of DNN relaxations for POP \eqref{eq:BBCPOP}. Specifically, if $\Ibox=\emptyset$, our hierarchy 
may be regarded 
as a variant of the sparse SDP relaxation proposed in \cite{WAKI2006,SparsePOP_UG} 
if the SDP is replaced by a stronger DNN relaxation
at each hierarchy level, although their formulations may look quite different. Therefore, if $\Ibox=\emptyset$, the convergence of the optimal value $\zeta$ of  COP \eqref{eq:relaxOfPOP} to the optimal value 
$\zeta^*$ of POP \eqref{eq:BBCPOP} is guaranteed \cite{LASSERRE2001b,LASSERRE2006}
as the relaxation order $\omega$ increases to infinity.

We also mention that COP \eqref{eq:relaxOfPOP} can be strengthened by replacing the condition 
``if $\r(c(\balpha+\bbeta)) = \r(\balpha'+\bbeta')$ for some $c\geq 1$'' with the condition 
``if $\r(\balpha+\bbeta) \leq \r(\balpha'+\bbeta')$'' in \eqref{eq:box} and in the description of $\coneK_2$. 
But we should take note that
 this replacement considerably increases the number of inequalities in $\coneK_2$ and
makes the computation of the metric projection $\Pi_{\coneK_2}(\cdot)$ to be very complicated and expensive.
As a result, 
APG Algorithm (hence BP Algorithm too)  is not expected to be efficient for the  COP with the
strengthened $\coneK_2$; see \cite{Ito2018a} for the details.

\section{Improving the solution quality} 

\subsection{An upper bound $\rho$ of the trace of the moment matrix} \label{sec:SubmoForRho}

Under the assumptions (A1) and (A2),  a valid lower bound $y_0^{v\ell}$ 
for the optimal value $y_0^*$ of COP \eqref{eq:generalCOP} is obtained  from 
BP Algorithm. 
The quality of the valid lower bound $y_0^{v\ell}$ 
depends noticeably  on the choice 
of $\rho > 0$ satisfying (A1) where a smaller $\rho$ will lead to a larger $y_0^{v\ell}$. In this section, we discuss 
the problem of minimizing $\rho$ among the ones that satisfy (A1). 

Throughout this section, we assume that 
COP \eqref{eq:generalCOP} is constructed for a conic relaxation of POP \eqref{eq:BBCPOP} 
as  described in Section 3. Thus, $y_0^{v\ell}$ serves as a valid lower bound of  the POP.
In this case, $\coneK_1$ turns out to be $\SymMat_+^{\AC^1_\omega}\times\SymMat_+^{\AC^2_\omega}\times\ldots\times\SymMat_+^{\AC^\ell_\omega}$. 
As a result,  $\I = (\I_1,\ldots,\I_{\ell})$ can be taken 
as an interior point of $\coneK_1$, where $\I_k$ denotes the identity matrix in $\SymMat^{\AC^k_\omega}$. 
The problem under consideration is written as 
\begin{align}
\hat{\rho} 
    & = \max \left\{\inprod{\I}{\Z} \mid \Z \ \text{ is a feasible solution of COP \eqref{eq:generalCOP}} \right\} \nonumber \\
    & = \max \left\{\sum_{k=1}^{\ell}\inprod{\I_k}{\Z_k} \Big| \inprod{\H}{\Z}=1, \ \Z = (\Z_1,\ldots,\Z_{\ell})\in\coneK_1\cap\coneK_2 \right\}. 
\label{eq:COPrho}
\end{align}
We may regard
the problem \eqref{eq:COPrho} 
as a DNN relaxation of the following POP:
\begin{align}
\rho^* 
  & = \max \left\{\sum_{k=1}^{\ell}\inprod{\I_k}{\x^{ \AC^k_\omega\times \AC^k_\omega}} \Big| \x \in H \ \text{(\textit{i.e.}, a feasible solution of POP \eqref{eq:BBCPOP})} \right\} \nonumber \\
  & = \max \left\{
    \sum_{k=1}^{\ell} \sum_{\balpha\in\AC^k_\omega}\x^{\balpha+\balpha}
  \Big| 
    \begin{array}{l}
      x_i \in [0,1] \ (i\in\Ibox), \ x_j \in \{0,1\} \ (j \in\Ibinary), \\
      \x^{\bgamma} = 0 \ (\bgamma \in \Gamma)
    \end{array}
  \right\}.
\label{eq:POPrho} 
\end{align} 
This implies that $\hat{\rho} \geq \rho^*$, and  if $\rho \geq \rho^*$ and $(y_0,\Y_2,\mu)$ 
is a feasible solution of COP 
\eqref{eq:generalDualCOPIX}, then $y_0 + \rho\mu$ provides a valid lower bound  for 
the optimal value of POP \eqref{eq:BBCPOP}. Thus, it is more reasonable to consider 
\eqref{eq:POPrho} directly  than its relaxation \eqref{eq:COPrho}. 
It is easy to verify that the problem \eqref{eq:POPrho} has an optimal solution $\x$ 
with $x_i \in \{0, 1\}$ 
for all $i=1,.\ldots,n$. In addition, if $\balpha \geq \bgamma$ for some $\bgamma \in \Gamma$ 
then $\x^{\balpha+\balpha}=0$. 
Hence, \eqref{eq:POPrho} is reduced to a combinatorial optimization problem given by
\begin{eqnarray}
\rho^* & = & \displaystyle \max \left\{ \sum_{k=1}^{\ell} \sum_{\balpha \in \BC_{\omega}^k} \x^{\balpha+\balpha} 
\Big| \x \in \{0,1\}^n, \ \x^{\bgamma} = 0 \ (\bgamma \in \Gamma) 
\right\},\label{eq:POPrho2}
\end{eqnarray}
where $\BC_{\omega}^k = \{ \balpha \in \AC_{\omega}^k \mid \balpha  \not\geq \bgamma \ \mbox{for any } \bgamma \in \Gamma \}$
 $(k=1,\ldots,\ell)$. 
It would be ideal  to use $\rho = \rho^*$ in BP Algorithm for a tight lower bound 
$y_0^{v\ell}$.  As  the problem \eqref{eq:POPrho2} is  numerically intractable in general,
any upper bound $\rho$ for $\rho^*$ can be used in practice. In particular,  the trivial upper bound 
$\rho = \sum_{k=1}^{\ell} \left|\BC_{\omega}^k\right|$ may be used for an upper bound of $\rho^*$, 
but it may not be tight  except for simple cases.

We can further reduce the problem \eqref{eq:POPrho2} to a submodular function minimization under
a set cover constraint for which efficient approximation algorithms \cite{Iwata2009,Wan2010}  
exist for a tight lower bound of the 
its minimum value.
For this purpose,  the vector variable $\x \in \{0, 1\}^n$ is replaced by a set variable 
$S_0 \subseteq N = \{1,\ldots,n\}$, which determines $x_i$ to be $0$ if $i \in S_0$ and $1$ otherwise. 
For every $i\in N$, define
$$
  E_i = \bigcup_{k=1}^{\ell} \{(k,\balpha) \mid \balpha \in \BC^k_\omega, ~\alpha_i\geq 1\}, \quad
  F_i = \{\bgamma \in \Gamma \mid \gamma_i \geq 1\}. 
$$
The next Lemma states two properties of the above sets.

\begin{lemma} \label{lemma:submodular} 
Choose $\x \in \{0,1\}^n$ arbitrarily. Let $S_0 = \{i\in N \mid x_i = 0\}$. Then we have that
\begin{description}
\item{(i) } $\bigcup_{i\in S_0} E_i=\bigcup_{k=1}^{\ell}\{(k,\balpha)\mid\balpha\in\BC^k_\omega, ~ \x^{\balpha}= 0\}$;
\item{(ii) } 
$\x^{\bgamma} = 0$ $(\bgamma \in \Gamma)$ if and only if 
$\bigcup_{i\in S_0} F_i = \Gamma$ or $\Big|\bigcup_{i\in S_0} F_i\Big| = | \Gamma |$. 
\end{description}
\end{lemma} 
\begin{proof}
(i) Assume that $(k,\balpha) \in E_i$ for some $i \in S_0$. By definition, $(k,\balpha) \in \BC^k_{\omega}$ and 
$\alpha_i \geq 1$. Hence $0=x_i=x_i^{\alpha_i} = \x^{\balpha}$. Thus we have shown the inclusion 
$\bigcup_{i\in S_0} E_i \subseteq \bigcup_{k=1}^{\ell}\{(k,\balpha)\mid\balpha\in\BC^k_\omega, ~ \x^{\balpha}= 0\}$. 
Now assume that $\balpha\in\BC^k_\omega$ and $\x^{\balpha}= 0$. It follows from $\x^{\balpha}= 0$ and 
$\x \in \{0,1\}^n$ that $\alpha_i \geq 1$ and $x_i = 0$ for some $i \in N$. 
Hence $i \in S_0$ and $(k,\balpha)\in\bigcup_{i\in S_0} E_i$. Thus we have shown the converse inclusion. \\
\mbox{ \  } \hspace{5mm} (ii) 
Assume that $\x^{\bgamma} = 0$ $(\bgamma \in \Gamma)$. 
The inclusion $\bigcup_{i\in S_0} F_i \subseteq \Gamma$ is straightforward by definition. If $\bgamma \in \Gamma$, 
then it follows from $\x^{\bgamma} = 0$ that $\gamma_i \geq 1$ and $x_i = 0$ for some $i \in N$; 
hence $i \in S_0$ and $\bgamma \in F_i$. Thus we have shown the converse inclusion, and 
the ``only if'' part of (ii).   Now 
assume that $\bigcup_{i\in S_0} F_i = \Gamma$. Let $\bgamma \in \Gamma$. Then there is an $i \in S_0$ such that 
$\bgamma \in F_i$: Hence $x_i = 0$ and $\gamma_i \geq 1$, which implies that $0=x_i = x_i^{\gamma_i} = \x^{\bgamma}$.  
Thus we have shown the ``if'' part of (ii). 
\end{proof}

By (ii) of Lemma~\ref{lemma:submodular}, we can rewrite the constraint of the problem \eqref{eq:POPrho2}  
as $S_0 \subseteq N$ 
and $\Big|\bigcup_{i\in S_0} F_i\Big| = | \Gamma |$, and the objective function as
\begin{eqnarray*}
\sum_{k=1}^{\ell} \sum_{\balpha\in\BC^k_\omega}\x^{\balpha+\balpha} & = & 
\sum_{k=1}^{\ell} \sum_{\balpha\in\BC^k_\omega} \sum_{\x^{\salpha }= 1} 1 
\ \mbox{(since $\x^{\balpha} = 0$ or $1$ for every $\balpha \in \BC^k_{\omega}$)}\\ 
& = & \sum_{k=1}^{\ell} \left( \left| \BC^k_{\omega} \right| 
- \sum_{\balpha\in\BC^k_\omega} \sum_{\x^{\salpha }= 0} 1 \right) 
 \\ 
& = & \sum_{k=1}^{\ell} \left| \BC^k_{\omega} \right| - \sum_{k=1}^{\ell}\sum_{\balpha\in\BC^k_\omega} \sum_{\x^{\salpha }= 0} 1 \\
& = &  \sum_{k=1}^{\ell} \left| \BC^k_{\omega} \right|  -  \left|\bigcup_{k=1}^{\ell}\{(k,\balpha)\mid\balpha\in\BC^k_\omega, ~ \x^{\balpha}= 0\}\right| \\ 
& = &  \sum_{k=1}^{\ell} \left| \BC^k_{\omega} \right|  - \left| \bigcup_{i\in S_0} E_i\right| \ \mbox{(by (i) of Lemma~\ref{lemma:submodular})}. 
\end{eqnarray*}
Therefore the problem 
\eqref{eq:POPrho2} is equivalent to the problem
\begin{equation}
    \min \{c(S_0) \mid S_0\subseteq N, \ f(S_0) = |\Gamma|\}, \label{eq:max_trace_submo}
\end{equation}
where $c$ and $f$ are \textit{submodular} functions defined by 
\begin{eqnarray*}
  c(S_0) = \Big|\bigcup_{i\in S_0} E_i\Big| \quad \text{and} \quad f(S_0) = \Big|\bigcup_{i\in S_0} F_i\Big| \ 
\mbox{for every } S_0 \subseteq N.
\end{eqnarray*}
This problem is known as a submodular minimization problem under a submodular cover constraint. 
Approximation algorithms \cite{Iwata2009,Wan2010} can be used to obtain a lower bound $\bar{c} \geq 0$ 
for the optimal value of \eqref{eq:max_trace_submo}. By construction, $\rho = \sum_{k=1}^{\ell} \left| \BC_{\omega}^k\right|-\bar{c}$  
provides an upper bound of \eqref{eq:POPrho2}, 
which is tighter than or equals to  $\sum_{k=1}^{\ell} \left|\BC_{\omega}^k\right|$. See \cite{Ito2018a} for the details.

\subsection{Enhancing APG Algorithm}

Although APG Algorithm has the strong theoretical complexity result such that $f(\Y^k_2) - f^* \leq O(1/k^2)$, 
in this subsection, we will propose some enhancements to improve its practical performance.
 We begin by noting that 
the term $\frac{(t_k) - 1}{t_{k+1}} (\Y_1^{k} - \Y_1^{k-1})$ in the substitution 
$\overline{\Y}_1^{k+1} \leftarrow \Y_1^{k} + \frac{(t_k) - 1}{t_{k+1}} (\Y_1^{k} - \Y_1^{k-1})$   
can be seen as the 
momentum of the sequence, and the monotonically increasing sequence $\{\frac{(t_k) - 1}{t_{k+1}}\}_{k=0}^\infty\subseteq [0,1)$ determines the amount of the momentum. 
When the momentum is high, the sequence $\{\Y_1^k\}_{k=0}^\infty$ would overshoot and oscillate around the optimal solution. 
In order to avoid such an oscillation to further speed up the convergence, 
we incorporate the adaptive restarting technique \cite{Candes2015} that resets the momentum back to zero ($t_k \leftarrow 1$) and takes a step back to the previous point $\Y_1^{k-1}$ when the objective value increases, i.e., $\|\X^{k}\|-\|\X^{k-1}\| > 0$. 
In order to avoid frequent restarts, we also employ the technique of \cite{Monteiro2016,Ito2017a} that
prohibits the next restart for $K_i$ iterations after the $i$th restart has occurred, where $K_0\geq 2$ 
and $K_i=2K_{i-1}$ $(i=1,2,\ldots)$. More precisely, we modify APG Algorithm to the following APGR Algorithm.

\medskip

\noindent
\rule[0mm]{162mm}{0.5mm}\\
{\bf APGR Algorithm} (Accelerated Proximal Gradient algorithm with restarting for feasibility test)\\
\rule[2mm]{162mm}{0.2mm}\vspace{-3mm}
  \begin{algorithmic}
    \STATE Input: $\G \in \XC$, $\Y_1^0 \in \XC$, $\Pi_{\coneK_1}$, $\Pi_{\coneK_2}$, $\epsilon>0$, $\delta>0$, $k_{max}>0$, $\eta_r>1$ 
    \STATE Output: $(\X^{k},\Y_1^{k},\Y_2^{k})$
    \STATE Initialize: $t_1 \leftarrow 1, ~L_1 \leftarrow 0.8, ~\overline{\Y}_1^{0} \leftarrow \Y_1^0$, $K_0=2$, $i=1$, $k_{re} = 0$ 
    \FOR{$k = 1,\dots,k_{max}$}
      \STATE $\Y_1^k \leftarrow \Pi_{\coneK_1^*}\Big(\overline{\Y}_1^k - \frac{1}{L_k}\Pi_{\coneK_2}(\overline{\Y}_1^k - \G)\Big)$  
      \STATE $\Y_2^{k} \leftarrow \Pi_{\coneK_2^*}(\G - \Y_1^{k})$, ~ $\X^{k} \leftarrow \G - \Y_1^{k} - \Y_2^{k}$ 
      \IF{ $\|\X^{k}\|<\epsilon$ or $g(\X^k,\Y_1^k,\Y_2^k) < \delta $  }
        \STATE \textbf{break}
      \ENDIF
      \STATE $t_{k+1}  \leftarrow  \frac{1+\sqrt{1+4t_k^2}}{2}$  
      \STATE $\overline{\Y}_1^{k+1} \leftarrow \Y_1^{k} + \frac{(t_k) - 1}{t_{k+1}} (\Y_1^{k} - \Y_1^{k-1})$  
      \IF{$\|\X^{k}\|-\|\X^{k-1}\|>0$ and $k>K_i+k_{re}$}
        \STATE $t^{k+1} \leftarrow  1$, $\overline{\Y}_1^{k+1} \leftarrow  \Y_1^k$, $k_{re} \leftarrow  k$ 
        \STATE $K_{i+1} \leftarrow  2K_i$, $i\leftarrow i+1$ 
        \STATE $L_{k+1}\leftarrow \eta_r L_k$ 
      \ELSE
        \STATE $L_{k+1}\leftarrow L_k$ 
      \ENDIF
    \ENDFOR
  \end{algorithmic}
\rule[4mm]{162mm}{0.5mm}\\

Recall that $L=1$ is a Lipschitiz constant for the gradient 
of the function $f$ defined in \eqref{eq:regressionFISTA}.
For APGR Algorithm, the convergence complexity result of
$f(\Y_1^k)-f^* \leq O((\log k/ k)^2)$ is ensured by \cite{Ito2017a,Ito2018a}. 
Although it is slightly worse than the original theoretical convergence guarantee  of $O(1/k^2)$, 
it often converges much faster in practice.
To improve the practical performance of APGR Algorithm, 
the initial  estimate $L_1$ of $L$ in the algorithm is set to a value less than $1$. 
For many instances,  $L_1 =0.8$ provided good results.

Even with the aforementioned practical improvements,  APGR Algorithm can still take a long time 
to compute
 a very accurate solution of the problem \eqref{eq:regressionFISTA1cone} 
 for the purpose of deciding whether $f^*=0$. 
If there exists sufficient evidence to show 
that the optimum value $f^*$ is not likely to be $0$, then 
APGR Algorithm can be terminated earlier to save computation time.  
To be precise, next
we discuss the stopping criteria of APGR Algorithm.
Let $g^k$ denote the violation $g(\X^k,\Y_1^k,\Y_2^k)$ of the KKT condition. 
Assume that $g^k$ is sufficiently small, i.e., the solution is nearly optimal. 
Then the ratio $\|\X^k\|/g^k$ of the optimal value and the KKT violation
is a reasonable measure to indicate
how far away $f^*$ is from 0. 
To determine that the value $\|\X^k\|/g^k$ will not be improved much,
 the following values are computed and tested at every  $k$th iteration ($k>30$):
\begin{itemize}
  \item $M_g^k=$ Geometric mean of $\{g^{k-i}/g^{k-i-30}\mid i=0,1,\ldots,9\}$. 
  \item $M_X^k=$ Geometric mean of $\{\|\X^{k-i}\|/\|\X^{k-i-30}\|\mid i=0,1,\ldots,9\}$.
\end{itemize}
If the above values are close to $1$,
much improvement in the KKT residual and objective values (and $\|\X^k\|/g^k$) cannot be expected.
Based on these observations, we implemented several heuristic stopping criteria in 
our practical implementation of APGR Algorithm in the {\sc Matlab} function called
\texttt{projK1K2\_fista\_dualK.m}. 
For example, we have found the following stopping criterion
\[  \|\X^k\|/g^k\geq 10^4, \quad g^k\leq \sqrt{\delta}, \quad M_g^k \geq 0.95, \quad M_X^k \geq 0.995\]
is useful to reduce the computational time substantially without losing the 
quality of our computed valid lower bound $y^{vl}_0$ from BP Algorithm.

\section{Numerical experiments} \label{sec:expBP}

We present numerical results to compare
the performance of \matBP \ with SDPNAL+ \cite{YST2015}, 
a state-of-the-art solver for large scale semidefinite optimization problem with nonnegative constraints. 
Recall that \matBP \ is a MATLAB implementation of the DNN relaxation~\eqref{eq:relaxOfPOP} of POP~\eqref{eq:BBCPOP} 
based on BP and APGR Algorithms, which have been presented in Sections 2.3 and 4.2, respectively.
SDPNAL+ is an augmented Lagrangian based method for which the main 
subproblem in each iteration is solved by a  semismooth Newton method that employs the 
generalized Hessian of the underlying function. As already mentioned in the Introduction, the solver is able to 
handle nondegenerate problems efficiently but it is usually not efficient in solving degenerate problems and generally
unable to solve such problems to high accuracy.
The solver \matBP \ on the other hand,  uses only  first-order derivative information and is specifically
designed to handle degenerate problems arising from the DNN relaxations of BBC constrained POPs. As a result,
SDPNAL+ is expected to provide a more accurate solution than \matBP\  when the former is successful
in solving the problem to the required accuracy.
In Section~2.3, we have presented a method to compute a valid lower bound  in BP Algorithm 
by introducing a primal-dual 
pair of COPs~\eqref{eq:generalCOPIX} and \eqref{eq:generalDualCOPIX}. 
Similarly, we also can generate a valid lower bound from a solution of SDPNAL+. 
For fair comparison, we also computed valid lower bounds based on the approximate 
solutions generated by SDPNAL+ for all experiments. 

For BP Algorithm incorporated in \matBP, the parameters $(tol,\gamma,\epsilon,\delta,k_{\max},\eta_r)$ 
were set to $(10^{-5},10,10^{-13},10^{-6},20000,1.1)$.
For SDPNAL+, the parameter ``tol'' was set to $10^{-6}$ and ``stopoptions'' to $2$ so that the solver continues to run even if it encounters some stagnations. 
SDPNAL+ was terminated in $20000$ iterations even if the stopping criteria were not satisfied. 
All the computations were performed in MATLAB on a Mac Pro with Intel Xeon E5 CPU (2.7 GHZ) and 64 GB memory. 

In Tables~\ref{table:RandBinNoEq} through~\ref{table:exp_QAP},  the meaning of the notation are as follows:
``opt'' (the optimal value), ``LBv'' (a valid lower bound), ``sec'' (the computation time in second), 
``apgit'' (the total number of iterations in APG Algorithm), ``bpit'' (the number of iterations in BP Algorithm), 
``iter'' (the number of iterations in SDPNAL+), and ``term'' (the termination code). 
The termination code of SDPNAL+ has the following meaning: 0 (problem is solved to required tolerance),
-1,-2,-3 ( problem is partially solved with the primal feasibility,  the dual feasibility, both feasibility 
slightly violating the required tolerance, respectively),
%
%
1 (problem is not solved successfully due to stagnation), 
2 (the maximum number of iterations reached).
For \matBP,  its termination code  has the following meaning: 1,2 (problem is solved to required tolerance), 3 (the iteration 
is stopped due to minor improvements in valid lower bounds).

\subsection{Randomly generated sparse polynomial optimization problems with binary, box and complementarity 
constraints}

In this section, we present
 experimental results on randomly generated sparse POPs 
of the form \eqref{eq:BBCPOP}.
The objective function $f_0(\x)=\sum_{\balpha\in \FC}c_{\balpha}\x^{\balpha}$ 
was generated as follows. 
Its degree was fixed to $2, \ 3, \ 4, \ 5, \ 6,$ and $8$. 
For the support supp$f=\FC$, we took 
\[
 \FC = \bigcup_{k=1}^{\ell} \Big\{\balpha\in\Integer_+^n ~\Big|~ \sum_{i=1}^n\alpha_i \leq d, \quad \alpha_i=0 ~~\text{if}~~ i\not\in V^k  \Big\}, 
\] 
where $V^k$ $(k=1,\ldots,\ell)$ were chosen from the two types of graphs, the arrow and chordal graphs, given 
in Section 3.2; see Figure~\ref{fig:spy}. 
Each $c_{\balpha}$ was chosen from the uniform distribution over $[-1,1]$. 
$\Ibinary$ was set to $N=\{1,2,\ldots, n\}$ (hence $\Ibox=\emptyset$) in Table~\ref{table:RandBinNoEq},
and $\emptyset$ (hence $\Ibox=N$) in Table~\ref{table:RandBoxNoEq}. In each table,
$\CC$ was set to $\emptyset$ and  a set of $2n$ elements randomly chosen 
from $\bigcup_{k=1}^\ell (V^k\times V^k)$, respectively. 
For all computation, the relaxation order $\omega$ was set to $\lceil\frac{d}{2}\rceil$.

\begin{landscape}
\noindent
\begin{table}[htbp]
\scriptsize{
  \centering
  \caption{  The computation results of SDPNAL+ and \matBP \ for $\Ibinary=\{1,2,\ldots,n\}$ and $\Ibox=\emptyset$.
  } \label{table:RandBinNoEq}
  \begin{tabular}{|c|r|r|r@{~(}r@{,~}r@{:}r@{,~}r@{)~~~}r@{~(}r@{,~}r@{,~}r@{)~}|r@{~(}r@{,~}r@{:}r@{,~}r@{)~~~}r@{~(}r@{,~}r@{,~}r@{)~}|}\hline
  \multicolumn{3}{|c|}{obj. $\backslash$ constr.}  & \multicolumn{9}{c|}{$\CC=\emptyset$} & \multicolumn{9}{c|}{$\CC:$ randomly chosen} \\ \hline
   &  &  & \multicolumn{5}{c}{\matBP} & \multicolumn{4}{c|}{SDPNAL+} & \multicolumn{5}{c}{\matBP} & \multicolumn{4}{c|}{SDPNAL+} \\ 
type & d & $n(\ell)$ & \multicolumn{5}{r}{LBv (sec,apgit:bpit,term)} & \multicolumn{4}{r|}{LBv (sec,iter,term)} & \multicolumn{5}{r}{LBv (sec,apgit:bpit,term)} & \multicolumn{4}{r|}{LBv (sec,iter,term)} \\ \hline \hline
\multirow{9}{*}{\shortstack{Arrow\\($a=10$, \\$b=2$, \\$c=2$)}}
  & 2 & 1284(160) & -9.115078e2 & 2.31e2 & 10517 & 16 & 2 & -9.114681e2 & 3.65e3 & 5597 & 0 & -3.845385e2 & 2.11e2 & 9583 & 16 & 2 & -3.845080e2 & 7.65e3 & 20000 & -2\\
  & 2 & 1924(240) & -1.365696e3 & 6.71e2 & 20249 & 16 & 2 & -1.365674e3 & 6.67e3 & 7740 & 0 & -5.620671e2 & 3.69e2 & 12016 & 17 & 2 & -5.620345e2 & 1.39e4 & 20000 & -2\\
  & 2 & 2564(320) & -1.848387e3 & 3.86e2 & 9188 & 17 & 2 & \multicolumn{4}{c|}{--} & -7.870155e2 & 9.57e2 & 23251 & 18 & 2 & \multicolumn{4}{c|}{--} \\ \cline{2-21}
  & 3 & 484(60) & -6.612489e2 & 1.12e3 & 12533 & 20 & 2 & -6.612368e2 & 5.62e3 & 4269 & 0 & -1.843149e2 & 6.03e2 & 6663 & 18 & 2 & -1.843102e2 & 2.29e3 & 953 & 0\\
  & 3 & 964(120) & -1.387246e3 & 2.28e3 & 11989 & 18 & 2 & -1.387269e3 & 3.28e4 & 5230 & 0 & -3.781069e2 & 1.50e3 & 7500 & 18 & 2 & -3.781066e2 & 1.38e4 & 2333 & 0\\
  & 3 & 1444(180) & -1.996709e3 & 3.56e3 & 13913 & 20 & 2 & \multicolumn{4}{c|}{--} & -5.413793e2 & 1.87e3 & 7410 & 18 & 2 & \multicolumn{4}{c|}{--} \\ \cline{2-21}
  & 5 & 52(6) & -1.628047e2 & 2.48e3 & 14866 & 21 & 2 & -1.730651e2 & 7.62e3 & 20000 & 1 & -1.893068e1 & 1.09e3 & 6869 & 18 & 2 & -1.892955e1 & 2.92e2 & 320 & 0\\
  & 5 & 100(12) & -3.153501e2 & 4.15e3 & 12200 & 19 & 2 & -3.224262e2 & 1.39e4 & 2402 & 0 & -4.600732e1 & 2.47e3 & 7543 & 18 & 2 & -4.600496e1 & 2.23e3 & 789 & 0\\
  & 5 & 148(18) & -4.836598e2 & 6.75e3 & 14284 & 20 & 2 & \multicolumn{4}{c|}{--} & -6.020487e1 & 2.46e3 & 5495 & 19 & 2 & \multicolumn{4}{c|}{--} \\ \hline 
\multirow{6}{*}{\shortstack{Arrow\\($a=5$, \\$b=2$, \\$c=2$)}}
  & 6 & 70(22) & -9.487728e1 & 1.57e2 & 7730 & 19 & 2 & -9.487333e1 & 1.70e2 & 335 & 0 & -1.583319e1 & 7.99e1 & 4433 & 17 & 2 & -1.583321e1 & 2.50e1 & 155 & 0\\
  & 6 & 76(24) & -1.035035e2 & 2.25e2 & 10062 & 18 & 2 & -1.035003e2 & 2.89e2 & 967 & 0 & -1.639041e1 & 9.84e1 & 4872 & 18 & 2 & -1.638971e1 & 4.46e1 & 317 & 0\\
  & 6 & 82(26) & -1.112050e2 & 1.53e2 & 6407 & 19 & 2 & -1.112020e2 & 2.99e2 & 856 & 0 & -1.797953e1 & 1.29e2 & 5948 & 17 & 2 & -1.797902e1 & 2.75e1 & 154 & 0\\ \cline{2-21}
  & 8 & 64(20) & -8.471209e1 & 4.59e2 & 10985 & 18 & 2 & -8.471209e1 & 7.77e2 & 3878 & 0 & -1.446632e1 & 1.57e2 & 4503 & 17 & 2 & -1.446599e1 & 3.96e1 & 154 & 0\\
  & 8 & 70(22) & -9.238185e1 & 4.01e2 & 8686 & 19 & 2 & -9.238109e1 & 6.81e2 & 1558 & 0 & -1.767580e1 & 1.63e2 & 4272 & 17 & 2 & -1.767573e1 & 1.19e2 & 631 & 0\\
  & 8 & 76(24) & -9.720811e1 & 5.03e2 & 10009 & 19 & 2 & -9.720601e1 & 7.95e2 & 2019 & 0 & -1.717366e1 & 1.70e2 & 3996 & 17 & 2 & -1.717314e1 & 4.20e1 & 153 & 0\\ \hline \hline
\multirow{15}{*}{\shortstack{Chordal\\($r=0.1$)}}
  & 2 & 400 & -3.474744e2 & 3.60e2 & 12180 & 17 & 2 & -3.474432e2 & 7.94e3 & 2178 & 0 & -8.638335e1 & 7.34e2 & 24016 & 18 & 2 & -8.637568e1 & 2.21e4 & 20000 & -2\\
  & 2 & 800 & -1.205713e3 & 1.26e3 & 13868 & 18 & 2 & -1.205601e3 & 2.80e4 & 2334 & 0 & -1.972467e2 & 1.05e3 & 11445 & 19 & 2 & -1.972365e2 & 3.37e4 & 2116 & 0\\
  & 2 & 1600 & -3.750599e3 & 5.89e3 & 12502 & 18 & 2 & \multicolumn{4}{c|}{--} & -5.124576e2 & 8.89e3 & 20542 & 20 & 2 &  &  &  &  \\ \cline{2-21}
  & 3 & 200 & -2.065760e2 & 3.52e2 & 13101 & 18 & 2 & -2.065634e2 & 8.55e2 & 656 & 0 & -7.533718e1 & 1.54e2 & 5459 & 19 & 2 & -7.533653e1 & 1.72e2 & 158 & 0\\
  & 3 & 300 & -4.769478e2 & 1.17e3 & 11293 & 18 & 2 & -4.769380e2 & 6.51e3 & 983 & 0 & -8.210410e1 & 3.81e2 & 3617 & 18 & 2 & -8.210191e1 & 3.93e2 & 154 & 0\\
  & 3 & 400 & -9.067157e2 & 1.95e4 & 27303 & 19 & 2 & \multicolumn{4}{c|}{--} & -1.082068e2 & 4.44e3 & 5756 & 19 & 2 &  &  &  &  \\ \cline{2-21}
  & 5 & 100 & -1.097916e2 & 4.67e1 & 6843 & 17 & 2 & -1.097897e2 & 2.09e2 & 663 & 0 & -8.239900e1 & 4.48e1 & 6314 & 18 & 2 & -8.239926e1 & 6.08e1 & 486 & 0\\
  & 5 & 200 & -6.115672e2 & 1.00e3 & 8164 & 18 & 2 & -6.115882e2 & 1.55e4 & 1386 & 0 & -1.091617e2 & 4.84e2 & 4269 & 18 & 2 & -1.091602e2 & 4.07e2 & 156 & 0\\
  & 5 & 300 & -1.637701e3 & 4.56e4 & 28839 & 19 & 2 & \multicolumn{4}{c|}{--} & -1.381003e2 & 5.80e3 & 3855 & 18 & 2 &  &  &  &  \\ \cline{2-21}
  & 6 & 100 & -1.229437e2 & 3.17e1 & 4790 & 16 & 2 & -1.229360e2 & 6.72e1 & 171 & 0 & -1.000038e2 & 2.79e1 & 4242 & 17 & 2 & -1.000038e2 & 2.21e1 & 156 & 0\\
  & 6 & 150 & -3.343328e2 & 1.41e2 & 7499 & 18 & 2 & -3.343190e2 & 6.56e2 & 983 & 0 & -1.224431e2 & 7.60e1 & 4225 & 19 & 2 & -1.224399e2 & 1.46e2 & 166 & 0\\
  & 6 & 200 & -8.586312e2 & 2.14e3 & 16750 & 18 & 2 & -8.586086e2 & 1.26e4 & 1151 & 0 & -1.151008e2 & 4.89e2 & 3994 & 18 & 2 & -1.150950e2 & 8.88e2 & 159 & 0\\ \cline{2-21}
  & 8 & 100 & -2.384588e2 & 2.76e1 & 4051 & 15 & 2 & -2.384465e2 & 1.85e2 & 496 & 0 & -1.824317e2 & 4.23e1 & 6057 & 17 & 2 & -1.824326e2 & 3.59e1 & 165 & 0\\
  & 8 & 150 & -5.620558e2 & 1.92e2 & 8539 & 18 & 2 & -5.620407e2 & 8.13e2 & 1152 & 0 & -1.161660e2 & 9.62e1 & 4438 & 17 & 2 & -1.161638e2 & 8.16e1 & 158 & 0\\
  & 8 & 170 & -9.313252e2 & 1.06e3 & 13272 & 19 & 2 & -9.313161e2 & 9.53e3 & 1839 & 0 & -1.475300e2 & 4.09e2 & 5602 & 17 & 2 & -1.475289e2 & 3.96e2 & 318 & 0\\ \hline
  \end{tabular}
}  
\end{table}
\end{landscape}

\begin{landscape}
\begin{table}[htbp]
\scriptsize{
  \centering
\caption{The computation results of SDPNAL+ and \matBP \ for $\Ibinary=\emptyset$ and $\Ibox=\{1,2,\ldots,n\}$. 
} \label{table:RandBoxNoEq}
  \begin{tabular}{|c|r|r|r@{~(}r@{,~}r@{:}r@{,~}r@{)~~~}r@{~(}r@{,~}r@{,~}r@{)~}|r@{~(}r@{,~}r@{:}r@{,~}r@{)~~~}r@{~(}r@{,~}r@{,~}r@{)~}|}\hline
  \multicolumn{3}{|c|}{obj. $\backslash$ constr.}  & \multicolumn{9}{c|}{$\CC=\emptyset$} & \multicolumn{9}{c|}{$\CC:$ randomly chosen} \\ \hline
   &  &  & \multicolumn{5}{c}{\matBP} & \multicolumn{4}{c|}{SDPNAL+} & \multicolumn{5}{c}{\matBP} & \multicolumn{4}{c|}{SDPNAL+} \\ 
type & d & $n(\ell)$ & \multicolumn{5}{r}{LBv (sec,apgit:bpit,term)} & \multicolumn{4}{r|}{LBv (sec,iter,term)} & \multicolumn{5}{r}{LBv (sec,apgit:bpit,term)} & \multicolumn{4}{r|}{LBv (sec,iter,term)} \\ \hline \hline
\multirow{9}{*}{\shortstack{Arrow\\($a=10$, \\$b=2$, \\$c=2$)}}
  & 2 & 1284(160) & -9.629408e2 & 2.37e2 & 11521 & 19 & 2 & -9.629396e2 & 1.46e3 & 2151 & 0 & -4.440439e2 & 1.82e2 & 8714 & 18 & 2 & -4.440437e2 & 5.60e3 & 6475 & 0\\
  & 2 & 1924(240) & -1.432735e3 & 4.83e2 & 15027 & 17 & 2 & -1.432743e3 & 2.85e3 & 2283 & 0 & -6.704890e2 & 3.08e2 & 9502 & 18 & 2 & -6.704828e2 & 8.74e3 & 5829 & 0\\
  & 2 & 2564(320) & -1.948288e3 & 4.51e2 & 10379 & 18 & 2 & -1.948220e3 & 4.27e3 & 3622 & 0 & -8.826387e2 & 4.66e2 & 11886 & 18 & 2 & -8.826803e2 & 1.62e4 & 20000 & 0\\ \cline{2-21}
  & 3 & 164(20) & -2.815738e2 & 9.97e2 & 28142 & 19 & 2 & -2.814370e2 & 3.57e3 & 20000 & -1 & -9.442008e1 & 4.35e2 & 11791 & 18 & 2 & -9.441908e1 & 5.52e3 & 20000 & -3\\ 
  & 3 & 324(40) & -5.878139e2 & 2.90e3 & 40618 & 19 & 2 & -5.879896e2 & 1.08e4 & 20000 & -3 & -1.817788e2 & 9.68e2 & 13033 & 18 & 2 & -1.817602e2 & 1.62e4 & 18383 & 0\\
  & 3 & 484(60) & -9.342330e2 & 3.95e3 & 39757 & 21 & 2 & \multicolumn{4}{c|}{--} & -3.024086e2 & 1.71e3 & 16224 & 18 & 2 & \multicolumn{4}{c|}{--} \\ \cline{2-21}
  & 5 & 20(2) & -1.541292e2 & 4.11e3 & 30618 & 19 & 2 & -1.784141e2 & 8.18e3 & 20000 & 1 & -2.023543e1 & 1.32e3 & 9907 & 18 & 2 & -2.024220e1 & 4.36e3 & 20000 & 1\\
  & 5 & 36(4) & -2.773921e2 & 1.13e4 & 43006 & 18 & 2 & -2.841172e2 & 2.29e4 & 20000 & 1 & -4.620292e1 & 4.71e3 & 17307 & 19 & 2 & -4.622614e1 & 7.70e3 & 20000 & -3\\
  & 5 & 52(6) & -3.736479e2 & 2.08e4 & 55131 & 20 & 2 & \multicolumn{4}{c|}{--} & -6.892798e1 & 5.02e3 & 12720 & 20 & 2 & \multicolumn{4}{c|}{--} \\ \hline 
\multirow{6}{*}{\shortstack{Arrow\\($a=5$, \\$b=2$, \\$c=2$)}}
  & 6 & 28(8) & -1.543165e2 & 5.85e2 & 23268 & 19 & 2 & -1.615659e2 & 2.48e3 & 20000 & 1 & -2.575588e1 & 6.16e2 & 24285 & 18 & 2 & -2.574200e1 & 3.02e3 & 20000 & -3\\ 
  & 6 & 34(10) & -1.926331e2 & 4.70e2 & 15913 & 18 & 2 & -1.931423e2 & 5.12e3 & 20000 & 1 & -2.901087e1 & 3.48e2 & 11568 & 17 & 2 & -2.902573e1 & 3.50e3 & 20000 & -3\\
  & 6 & 40(12) & -2.229557e2 & 1.14e3 & 32413 & 20 & 2 & -2.229243e2 & 6.94e3 & 20000 & 1 & -3.887661e1 & 3.39e2 & 9359 & 18 & 2 & -3.887465e1 & 5.12e3 & 20000 & -3\\ \cline{2-21}
  & 8 & 16(4) & -2.062974e2 & 4.74e3 & 42025 & 22 & 2 & -2.064534e2 & 1.18e4 & 20000 & 1 & -7.338345e0 & 1.85e3 & 16555 & 19 & 2 & -7.335928e0 & 3.99e3 & 20000 & -2\\
  & 8 & 22(6) & -2.513190e2 & 5.41e3 & 52243 & 21 & 2 & -2.517759e2 & 3.44e4 & 20000 & 1 & -1.919867e1 & 2.57e3 & 14396 & 17 & 2 & -1.919410e1 & 1.85e4 & 20000 & -3\\
  & 8 & 28(8) & -3.020144e2 & 7.92e3 & 35596 & 18 & 2 & \multicolumn{4}{c|}{--} & -2.725916e1 & 1.95e3 & 8482 & 17 & 2 & \multicolumn{4}{c|}{--} \\ \hline \hline
\multirow{15}{*}{\shortstack{Chordal\\ ($r=0.1$)}}
  & 2 & 400 & -3.486683e2 & 5.04e2 & 17758 & 19 & 2 & -3.486410e2 & 1.02e4 & 2681 & 0 & -8.688482e1 & 3.34e2 & 11287 & 19 & 2 & -8.687822e1 & 1.30e4 & 10887 & 0\\
  & 2 & 800 & -1.206955e3 & 2.66e3 & 30865 & 18 & 2 & -1.206830e3 & 2.76e4 & 4012 & 0 & -1.975751e2 & 1.19e3 & 13331 & 18 & 2 & -1.975661e2 & 1.56e4 & 2561 & 0\\
  & 2 & 1200 & -2.385826e3 & 7.53e3 & 35674 & 18 & 2 & \multicolumn{4}{c|}{--} & -3.455533e2 & 3.04e3 & 14767 & 19 & 2 &  &  &  &  \\ \cline{2-21}
  & 3 & 100 & -5.668570e1 & 8.91e1 & 13436 & 19 & 2 & -5.668556e1 & 1.87e3 & 20000 & -3 & -5.304373e1 & 1.04e2 & 15051 & 20 & 2 & -5.304241e1 & 2.22e3 & 17843 & 0\\
  & 3 & 200 & -2.124543e2 & 1.26e3 & 37656 & 18 & 2 & -2.123967e2 & 1.46e4 & 20000 & -3 & -7.678323e1 & 4.25e2 & 12432 & 18 & 2 & -7.678226e1 & 1.66e4 & 13853 & 0\\
  & 3 & 400 & -9.481387e2 & 2.91e4 & 30987 & 20 & 2 & \multicolumn{4}{c|}{--} & -1.100166e2 & 2.72e4 & 29200 & 19 & 2 &  &  &  &  \\ \cline{2-21}
  & 5 & 80 & -6.679212e1 & 1.58e2 & 21489 & 17 & 2 & -6.678636e1 & 2.19e3 & 20000 & -3 & -6.470095e1 & 1.54e2 & 20670 & 18 & 2 & -6.470372e1 & 2.56e3 & 20000 & -3\\
  & 5 & 160 & -3.162245e2 & 1.49e3 & 24909 & 19 & 2 & -3.209856e2 & 2.29e4 & 20000 & 1 & -1.200387e2 & 1.26e3 & 21131 & 20 & 2 & -1.201353e2 & 2.46e4 & 20000 & -3\\
  & 5 & 240 & -1.056989e3 & 5.91e4 & 64866 & 19 & 2 & \multicolumn{4}{c|}{--} & -1.376493e2 & 1.25e4 & 14180 & 18 & 2 &  &  &  &  \\ \cline{2-21}
  & 6 & 50 & -5.433668e1 & 3.36e1 & 9673 & 17 & 2 & -5.433165e1 & 5.40e2 & 12140 & 0 & -5.237042e1 & 3.71e1 & 9883 & 19 & 2 & -5.236817e1 & 3.97e2 & 11692 & 0\\
  & 6 & 100 & -1.314831e2 & 2.76e2 & 24792 & 17 & 2 & -1.314177e2 & 4.17e3 & 20000 & -3 & -1.065003e2 & 2.63e2 & 23767 & 20 & 2 & -1.064412e2 & 4.73e3 & 20000 & -3\\
  & 6 & 120 & -2.130796e2 & 1.01e3 & 52465 & 19 & 2 & -2.133697e2 & 1.17e4 & 20000 & 1 & -1.122865e2 & 3.94e2 & 19744 & 19 & 2 & -1.122791e2 & 7.57e3 & 20000 & -3\\ \cline{2-21}
  & 8 & 50 & -3.542262e1 & 1.19e2 & 22885 & 18 & 2 & -3.540774e1 & 7.59e2 & 17063 & 0 & -3.305533e1 & 1.39e2 & 25186 & 18 & 2 & -3.303788e1 & 5.14e2 & 8524 & 0\\
  & 8 & 100 & -2.567746e2 & 9.45e2 & 31985 & 18 & 2 & -2.565093e2 & 8.29e3 & 20000 & 1 & -1.962108e2 & 4.56e2 & 15764 & 18 & 2 & -1.961334e2 & 8.47e3 & 20000 & -3\\
  & 8 & 120 & -3.106839e2 & 1.88e3 & 31566 & 21 & 2 & -5.129922e2 & 1.98e4 & 20000 & 1 & -1.392161e2 & 1.20e3 & 20163 & 18 & 2 & -1.397772e2 & 1.89e4 & 20000 & 1\\ \hline
  \end{tabular}
}  
\end{table}%
\end{landscape}

Tables~\ref{table:RandBinNoEq} and \ref{table:RandBoxNoEq}   show
the results on POPs with binary constraints and 
those with box constraints, respectively. 
We see that the lower bounds of the POP obtained by \matBP \ are comparable to those by SDPNAL+.
In particular, when the arrow type sparsity and $\Ibinary = \emptyset$ were used 
as in Table~\ref{table:RandBinNoEq}, 
\matBP \ could compute tighter lower bounds than SDPNAL+ in some instances. 
The computation time of \matBP \ is much smaller than SDPNAL+ in most instances. 
\matBP \ converged in $10^4  \sim  10^5$ seconds for very large scale POPs for which SDPNAL+ could not in $10^5$ seconds. 
Although \matBP \ uses only the first-order derivative information   contrary to SDPNAL+ that utilizes 
the second-order information, 
it could compute lower bounds comparable to SDPNAL+ in shorter computation time. 

As the number of variables  increases for the problems with $d = 2, 3, 5$ in 
Table~\ref{table:RandBinNoEq}, 
SDPNAL+ frequently failed to obtain a lower bound while \matBP \ succeeded to compute a valid 
lower bound.  As mentioned in Section 1, SDPNAL+ could not deal with degenerate POPs well. \matBP , 
on the other hand, provided valid lower bounds for those POPs. Similar observation can be made
for Table  \ref{table:RandBoxNoEq} where no lower bounds and computational time were 
reported for SDPNAL+ in some problems.
The numerical results in Tables~\ref{table:RandBinNoEq} and \ref{table:RandBoxNoEq}   demonstrate the robustness of \matBP \ against degeneracy.  We see from the results that
 the performance of \matBP \ is very efficient and robust for solving the DNN relaxations of large scale POPs with 
 BBC constraints.

For the ease of comparison,
in Figures \ref{fig-table1} and \ref{fig-table2}, we plot the relative gap of the valid lower bounds 
computed by SPDNAL+ and \matBP\, i.e.,
$\frac{{\rm LBv(BBCPOP)}-{\rm LBv(SDPNAL+)}}{|{\rm LBv(BBCPOP)}|}$,
and the ratios of the execution times taken by 
SDPNAL+ to those by \matBP\ for the instances tested in Tables~\ref{table:RandBinNoEq} and \ref{table:RandBoxNoEq},
respectively.
  Note that in the plots, a positive value for the relative gap of the lower bounds means that \matBP\ has computed a 
 tighter bound than SDPNAL+.
 We can observe that for the instances corresponding to $\mathcal{C} = \emptyset$ in Table 1 (Figure \ref{fig-table1}), \matBP\ 
 is a few times faster than SDPNAL+ in solving most of the instances, and it can be 10-20 times faster on a few of the 
 large instances. On the other hand, for the instances corresponding to a
  randomly chosen $\mathcal{C}$  in Table 1 (Figure \ref{fig-table1}), \matBP\ and SDPNAL+ have comparable
  efficiency in solving most of the instances. But for a  few smaller instances, 
  SDPNAL+ is 2--5  times faster while for some other larger instances, \matBP\ is at least 10-20 times faster. 
 
 For the instances in Table 2 (Figure \ref{fig-table2}), the efficiency of \matBP\ completely dominates that of SDPNAL+.
 Here, except for a few instances, \matBP\ is at least 5--20 times faster than 
 SDPNAL+ while the valid lower bounds generated by \matBP\ are comparable or much tighter than those generated by
 SDPNAL+.

\begin{figure}
 \centering
 \includegraphics[width=1.0\textwidth]{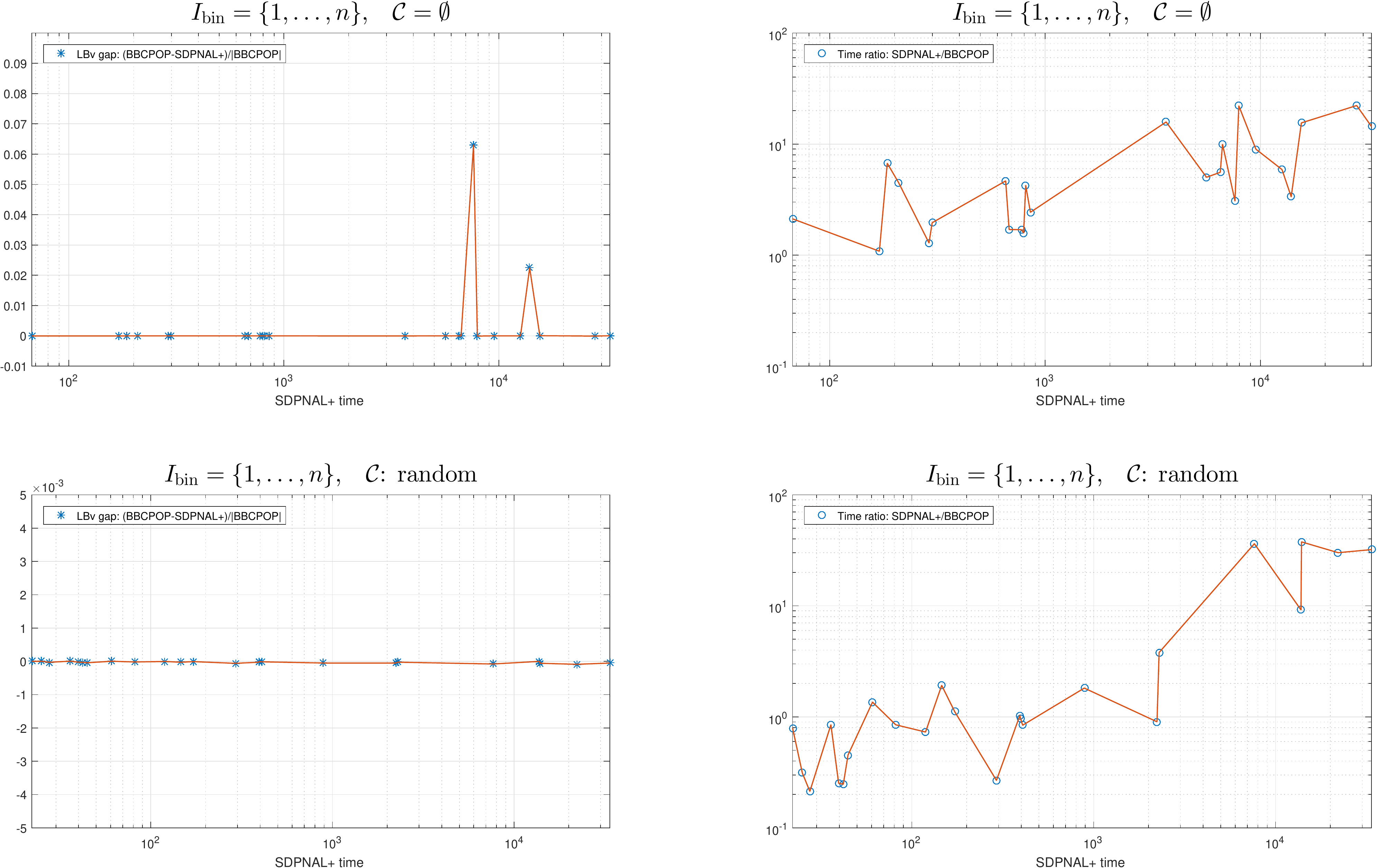}
 \caption{Comparison of the computational efficiency between \matBP\ and SDPNAL+ for 
 the instances (with $I_{\rm bin} = \{1,\ldots,n\}$ and  $I_{\rm box}  = \emptyset$)  in Table 1.
 }
 \label{fig-table1}
%
 \includegraphics[width=1.0\textwidth]{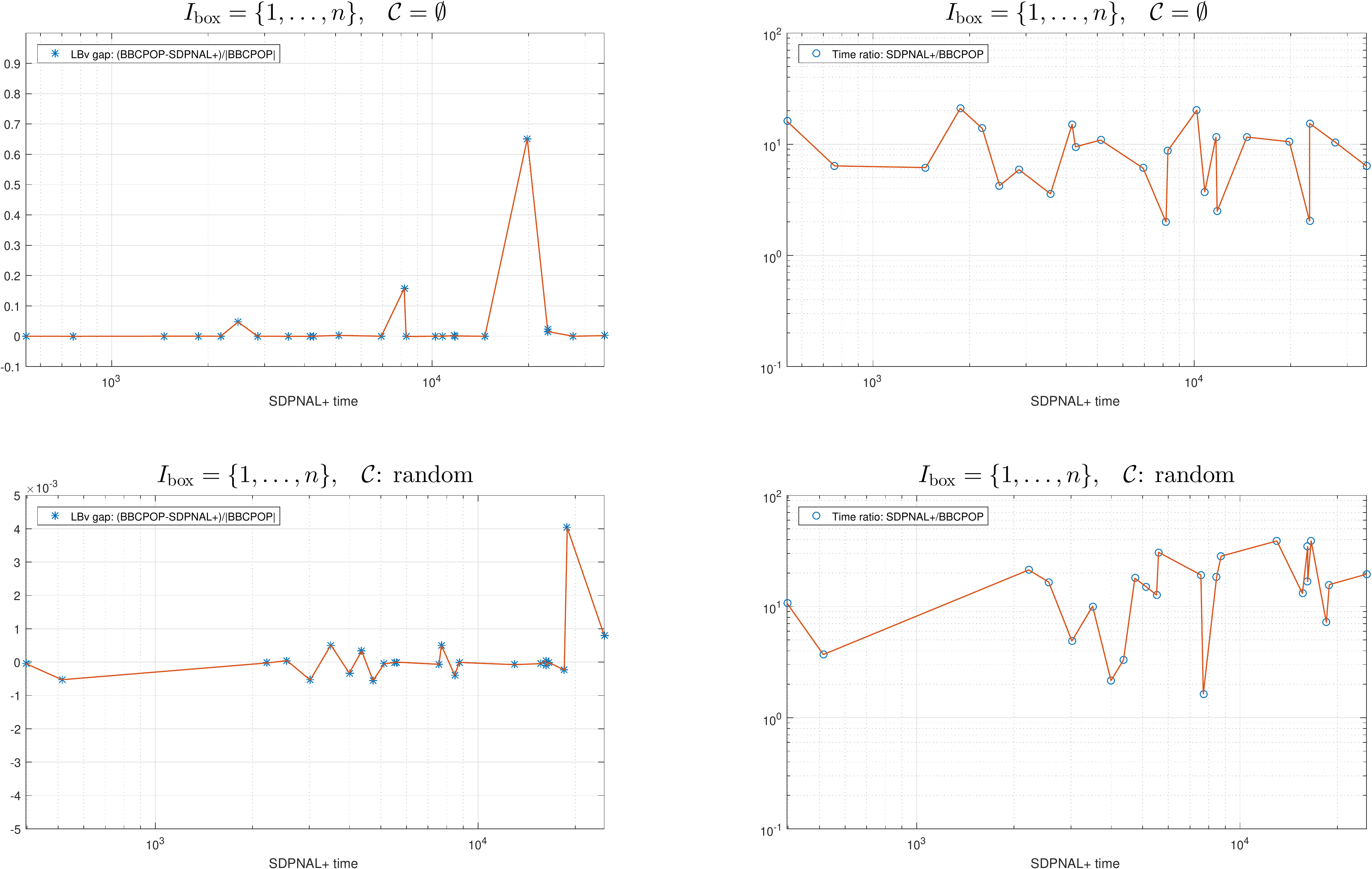}
 \caption{Comparison of the computational efficiency between \matBP\ and SDPNAL+ for 
 the instances (with $I_{\rm bin}  = \emptyset$ and $I_{\rm box} = \{1,\ldots,n\}$) in Table 2.
 }
 \label{fig-table2}
\end{figure}

\subsection{Quadratic Assignment Problem}

Various important combinatorial optimization problems such as the max-cut problem, the maximum stable set problem, and the quadratic assignment problem (QAP) are often formulated as quadratic optimization problems. 
In this section, we show the numerical performances of \matBP \ and SDPNAL+ in solving DNN relaxation problems of QAPs. 
We refer the readers to \cite{KIM2013,ARIMA2014b} for results on the max-cut and the maximum stable set problems.

Let $\A$ and $\B$ be given $r\times r$ matrices. 
Then the QAP is described as 
\begin{equation}
\zeta^*_{\text{QAP}} = \min \left\{
\inprod{\X}{(\A\X\B^T) }
\mid \X \text{ is a permutation matrix}
\right\}. \label{eq:QAPmat}
\end{equation}
Here we characterize a permutation matrix $\X$ as follows:
\begin{align*}
  &\X\in[0,1]^{r\times r}, ~~ \X\e = \X^t\e = \e,\\
  &X_{ik}X_{jk} = X_{ki}X_{kj} = 0 ~(i,j,k=1,2,\ldots,r; ~~ i \neq j).
\end{align*}
Let $\X=(\x^1,\ldots,\x^r) \in \Real^{r \times r}$, where $\x^p$ denotes the $p$-th column vector of $\X$. 
Let $n=r^2$. 
Arranging  the columns $\x^p$ $(1 \leq p \leq r)$  of $\X$ vertically into a long vector $\x=[\x^1;\ldots;\x^r]$,
we obtain the following reformulation of \eqref{eq:QAPmat}:
\begin{equation}
  \min\Big\{\x^T(\B\otimes\A)\x  ~\Big|~ 
  \begin{array}{l}
    \x\in[0,1]^n, \quad (\I\otimes\e^T)\x=(\e^T\otimes\I)\x=\e,\\
    x_ix_j = 0 \quad (i,j\in J_p; i\neq j; 1\leq p \leq 2r),
  \end{array}\Big\} \label{eq:QAPvec}
\end{equation}
where
\begin{align*}
  J_i &=\{(i-1)r+1,(i-1)r+2,\ldots,(i-1){r}+r\} \ (1 \leq i \leq r), \\
  J_{r+j} &= \{j,j+r,\ldots,j+(r-1)r\} \ (1 \leq j \leq r).
\end{align*}
Since it contains $2r$ equality constraints, 
the Lagrangian relaxation is applied to \eqref{eq:QAPvec}. 
More precisely, let $\C = [\I\otimes\e^T; \e^T\otimes\I]$ and $\d = [\e;\e]$, 
and consider the following Lagrangian relaxation problem of \eqref{eq:QAPvec} with a penalty parameter $\lambda> 0$:
\begin{equation}
  \min\Big\{\x^T(\B\otimes\A)\x + \lambda\frac{\|(\B\otimes\A)\|}{ \left\|\left(\begin{smallmatrix} \d^T\d & \d^t\C \\ \C\d & \C^T\C \end{smallmatrix}\right)\right\|} \|\C\x-\d\|^2 ~\Big|~ 
  \begin{array}{ll}
  \x\in[0,1]^n,\\ 
  x_ix_j=0 & (i,j\in J_p, i\neq j,\\
           &  1\leq p \leq 2r)
  \end{array}
  \Big\}, 
  \label{eq:QAPLag}
\end{equation}
which is a special case of POP~\eqref{POP0} or~\eqref{eq:BBCPOP}. Thus, we can apply the DNN relaxation 
discussed in Section 3 to \eqref{eq:QAPLag} and obtain a COP of the form~\eqref{eq:relaxOfPOP}, 
which serves as the Lagrangian-DNN relaxation \cite{KIM2013,ARIMA2017} of QAP~\eqref{eq:QAPvec}.
We set $\lambda=10^5$ for all instances, and solved it by \matBP \ and SDPNAL+. 
For comparison, SDPNAL+  was also applied to another DNN relaxation, AW+ formulation \cite{POVH2007}, 
which is a benchmark formulation 
for SDPNAL+ in the paper \cite{YST2015}. 
We mention that the AW+ formulation does not use the Lagrangian relaxation.

\begin{table}
\centering
\scriptsize
\caption{Computational results on small QAP instances. 
Both \matBP \ and SDPNAL+ were applied to 
the Lagrangian-DNN relaxation.
}\label{table:exp_smallQAP}
  \begin{tabular}{|c|r@{~(}r@{, }r@{:}r@{)~}|r@{~(}r@{, }r@{)~}|r@{~(}r@{, }r@{)~}|}\hline
     & \multicolumn{4}{c|}{\matBP} & \multicolumn{3}{c|}{SDPNAL+} & \multicolumn{3}{c|}{SDPNAL+ (AW+)} \\
instance & LBv & sec & apgit & bpit & LBv & sec & iter & LBv & time & iter \\\hline
chr12a & 9551.9 & 8.98e0 & 1984 & 16 & 9304.2 & 1.06e2 & 20000 & 9551.9 & 1.28e1 & 1553 \\
chr12b & 9741.8 & 1.05e1 & 2056 & 16 & 9669.5 & 1.17e2 & 20000 & 9741.9 & 1.10e1 & 1553 \\
chr12c & 11155.9 & 8.63e0 & 2042 & 17 & 10911.0 & 1.28e2 & 20000 & 11156.0 & 3.45e1 & 3591 \\
had12 & 1651.9 & 8.38e0 & 1488 & 16 & 1616.0 & 1.60e2 & 20000 & 1652.0 & 1.72e1 & 1522 \\
nug12 & 567.9 & 2.04e1 & 3701 & 15 & 565.6 & 1.78e2 & 20000 & 567.8 & 1.60e1 & 1054 \\
rou12 & 235521.1 & 4.62e1 & 6681 & 23 & 229651.9 & 1.44e2 & 20000 & 235518.8 & 6.47e1 & 5464 \\
scr12 & 31407.6 & 2.33e1 & 3843 & 28 & 30827.0 & 1.62e2 & 20000 & 31410.0 & 7.07e0 & 495 \\
tai12a & 224411.0 & 1.68e1 & 3154 & 22 & 219806.4 & 1.42e2 & 20000 & 224416.0 & 6.57e0 & 426 \\
tai12b & 39464040.0 & 1.47e1 & 2716 & 19 & 38647340.0 & 1.41e2 & 20000 & 39464910.0 & 1.95e1 & 2355 \\\hline
  \end{tabular}
\end{table}

In Table~\ref{table:exp_smallQAP}, we see that SDPNAL+ applied to the Lagrangian-DNN relaxation 
of QAP~\eqref{eq:QAPvec} shows inferior results compared to those obtained from \matBP.
This is mainly because the Lagrangian-DNN relaxation is highly degenerated.  
\matBP, on the other hand, could solve such ill-conditioned problems successfully, which demonstrates
the robustness of BP algorithm for ill-conditioned COPs. 
As the numerical results in \cite{YST2015} show that the AW+ formulation works well for SDPNAL+, 
it is used in the subsequent experiments on large scale problems.

\begin{table}
\scriptsize
\centering
\caption{Computational results on large-scale QAP instances. \matBP was applied to the Lagrangian-DNN relaxation, 
and SDPNAL+ to the AW+ formulation.
}
\label{table:exp_QAP}
  \begin{tabular}{|c||r|r@{ (}r@{, }r@{:}l@{, }r@{) }|r@{ (}r@{, }r@{, }r@{) }|} \hline
 & &\multicolumn{5}{c|}{\matBP} & \multicolumn{4}{c|}{SDPNAL+ (AW+)}  \\
instances & opt & LBv & sec & apgit & bpit & term & LBv & sec & iter & term \\ \hline\hline
chr15a & 9896.0 & 9895.9 & 3.78e1 & 2572 & 18 & 3 & 9892.2 & 2.21e2 & 6714 & 0 \\
chr15b & 7990.0 & 7989.9 & 3.25e1 & 2335 & 16 & 3 & 7989.8 & 5.25e1 & 2066 & 0 \\
chr15c & 9504.0 & 9503.9 & 3.11e1 & 2346 & 19 & 1 & 9504.0 & 4.26e1 & 1432 & 0 \\
chr18a & 11098.0 & 11097.9 & 1.27e2 & 3608 & 19 & 2 & 11088.3 & 2.77e2 & 6586 & 0 \\
chr18b & 1534.0 & 1532.5 & 2.76e2 & 6117 & 17 & 1 & 1533.9 & 5.57e1 & 917 & 0 \\
chr20a & 2192.0 & 2191.9 & 1.97e2 & 3264 & 17 & 1 & 2191.7 & 4.55e2 & 6402 & 0 \\
chr20b & 2298.0 & 2298.0 & 1.25e2 & 2194 & 17 & 1 & 2297.9 & 2.91e2 & 3412 & 0 \\
chr20c & 14142.0 & 14141.7 & 1.86e2 & 3352 & 20 & 2 & 14139.4 & 3.29e2 & 5689 & 0 \\
chr22a & 6156.0 & 6156.0 & 2.90e2 & 2926 & 19 & 1 & 6153.4 & 7.65e2 & 5791 & 0 \\
chr25a & 3796.0 & 3795.9 & 5.62e2 & 3139 & 18 & 1 & 3795.8 & 1.06e3 & 3866 & 0 \\ \hline
nug20 & 2570.0 & 2506.0 & 2.24e2 & 3024 & 18 & 1 & 2505.9 & 2.21e2 & 2056 & 0 \\
nug25 & 3744.0 & 3625.4 & 8.68e2 & 3911 & 19 & 1 & 3625.3 & 6.65e2 & 2053 & 0 \\
nug30 & 6124.0 & 5948.9 & 2.30e3 & 3634 & 21 & 1 & 5948.7 & 2.10e3 & 2300 & 0 \\ \hline
bur26a & 5426670.0 & 5426095.0 & 2.17e3 & 7280 & 20 & 2 & 5425904.2 & 3.33e3 & 9350 & 0 \\
bur26b & 3817852.0 & 3817277.4 & 2.14e3 & 7265 & 20 & 2 & 3817148.6 & 2.37e3 & 6608 & 0 \\
bur26c & 5426795.0 & 5426203.7 & 2.70e3 & 9049 & 21 & 2 & 5426457.7 & 6.19e3 & 18182 & 0 \\
bur26d & 3821225.0 & 3820014.6 & 1.90e3 & 6470 & 20 & 2 & 3820601.2 & 3.45e3 & 10800 & 0 \\
bur26e & 5386879.0 & 5386572.0 & 7.77e2 & 3110 & 21 & 2 & 5386585.4 & 3.52e3 & 11011 & 0 \\
bur26f & 3782044.0 & 3781834.8 & 6.45e2 & 2545 & 21 & 2 & 3781760.8 & 2.67e3 & 8568 & 0 \\
bur26g & 10117172.0 & 10116571.1 & 5.01e2 & 2117 & 20 & 2 & 10116504.8 & 2.23e3 & 8341 & 0 \\
bur26h & 7098658.0 & 7098236.7 & 6.10e2 & 2560 & 22 & 2 & 7098381.4 & 2.15e3 & 8728 & 0 \\ \hline
tai30a & 1818146.0 & 1706789.7 & 8.44e2 & 1357 & 21 & 2 & 1706814.3 & 2.18e3 & 2501 & 0 \\
tai30b & 637117113.0 & 598629428.0 & 4.90e3 & 7388 & 21 & 2 & 598979464.0 & 8.21e3 & 12701 & 0 \\
tai35a & 2422002.0 & 2216540.3 & 2.38e3 & 1482 & 22 & 2 & 2216573.2 & 3.33e3 & 1751 & 0 \\
tai35b & 283315445.0 & 269532369.0 & 8.65e3 & 6272 & 21 & 2 & 269624118.0 & 1.32e4 & 8153 & 0 \\
tai40a & 3139370.0 & 2843198.6 & 4.52e3 & 1223 & 21 & 2 & 2843245.1 & 1.90e4 & 2801 & 0 \\
tai40b & 637250948.0 & 608808415.0 & 3.04e4 & 8167 & 21 & 2 & 608955916.0 & 3.40e4 & 6794 & 0 \\
tai50a & 4938796.0 & 4390743.4 & 4.27e4 & 3537 & 21 & 2 & 4390862.9 & 4.82e4 & 2651 & 0 \\
tai50b & 458821517.0 & 431090745.0 & 4.86e4 & 5072 & 21 & 2 & 431074160.0 & 9.20e4 & 7300 & 0 \\ \hline
sko42 & 15812.0 & 15332.6 & 1.89e4 & 3898 & 21 & 2 & 15332.5 & 2.18e4 & 3593 & 0 \\
sko49 & 23386.0 & 22650.2 & 3.62e4 & 3383 & 21 & 2 & 22650.4 & 5.85e4 & 3517 & 0 \\ \hline
lipa40a & 31538.0 & 31536.5 & 9.24e3 & 3062 & 22 & 2 & 31538.0 & 1.44e4 & 3497 & 0 \\
lipa40b & 476581.0 & 476563.3 & 1.18e4 & 4405 & 21 & 2 & 476581.0 & 4.61e3 & 935 & 0 \\
lipa50a & 62093.0 & 62089.6 & 5.35e4 & 2061 & 21 & 2 & 62093.0 & 6.69e4 & 3099 & 0 \\
lipa50b & 1210244.0 & 1210195.2 & 6.84e4 & 2795 & 21 & 2 & 1210244.0 & 3.12e4 & 1554 & 0 \\ \hline
tho40 & 240516.0 & 226490.1 & 1.82e4 & 4826 & 21 & 2 & 226482.4 & 1.38e4 & 2500 & 0 \\
wil50 & 48816.0 & 48121.0 & 6.66e4 & 5453 & 21 & 2 & 48120.1 & 6.97e4 & 4178 & 0 \\ \hline
  \end{tabular}
\end{table}

The results for  large-scale QAPs are shown in Table~\ref{table:exp_QAP}.
For most problems, \matBP \ produced comparable lower bounds with SDPNAL+ 
while it terminated slightly faster than SDPNAL+ in many cases.
We note that the original QAP is not in the form of \eqref{POP0} to which \matBP \ can be applied. 
The results in Table~\ref{table:exp_QAP}, however, indicate that \matBP \ can solve the Lagrangian-DNN 
relaxations of very large-scale QAPs with high efficiency. 

Finally, we note that various equivalent DNN relaxation formulations for QAPs have been proposed. 
The performance of the BP method \ and SDPNAL+ are 
expected to differ from one formulation to another; see Section 7 of \cite{ITO2017c} for more  numerical 
results and investigation on the differences in the formulations and the performance of the
BP method and SDPNAL+.

\section{Concluding remarks}

We have introduced a Matlab software package, \matBP, to compute  valid lower bounds for the optimal values of
large-scale sparse POPs with binary, box and complementarity (BBC) constraints.
The performance of \matBP\    
has been illustrated with the numerical results  in Section 5 
on various {large-scale} sparse POP instances  with 
BBC constraints in comparison to those of SDPNAL+.

In terms of the number of variables, 
\matBP\ can handle efficiently larger POPs  than the other available
software packages and numerical methods for POPs such as   GlotiPoly  \cite{GLOPTIPOLY2003}, SOSTOOLs \cite{SOSTOOLS}, SparsePOP \cite{WAKI2008}, BSOS \cite{TOH17}, and SBSOS~\cite{WEISSER17}. 
\matBP\ 
not only automatically generate sparse DNN relaxations of a given
POP with BBC constraints,  it also provides a robust 
numerical method, BP Algorithm, specially designed for solving the DNN relaxation problems.
This is in contrast to 
the other software packages such as GloptiPoly, SparsePOP, 
SOSTOOLS, BSOS, and SBSOS that need to rely on an available SDP solver.
As a result, their performance depends on the SDP solver chosen.
One important feature of the lower bounds computed by \matBP\ is that it is theoretically guaranteed to be valid, 
whereas the lower bounds obtained by the other software packages, however tight they may be, is not
guaranteed to be valid unless the relaxation problem is solved to very high accuracy.

For general POPs with polynomial equality and inequality constraints, Lagrangian relaxations can be applied for
 \matBP.  
The Lagrangian relaxation approach was successfully used to solve combinatorial QOPs as shown in
 \cite{KIM2013} and in Section 5.2, and can be extended to general POPs.
General POPs are difficult problems to solve with a computational method. It has been our experience that
only the combination of SparsePOP \cite{WAKI2008} and the implementation of the primal-dual interior-point method, SeDuMi \cite{STURM99},
could successfully deal with  general POPs of moderate size. Other softwares including SDPNAL+  \cite{YST2015} have been unsuccessful to provide valid lower bounds for general POPs because the relaxation problems
are generally highly degenerate. As a future work,
we plan to  investigate the Lagrangian relaxation methods together with \matBP\ for solving DNN relaxations of
general POPs.

\bibliographystyle{plain}
\begin{small}

\end{small}
\end{document}